\documentclass[11pt,a4paper]{extarticle}
\usepackage{geometry}
\usepackage{stackrel}
\usepackage{verbatim}
\usepackage{enumerate}
\newcommand{\RN}[1]{%
  \textup{\uppercase\expandafter{\romannumeral#1}}%
}
\usepackage{xcolor}
\usepackage{subfig}
\usepackage{graphicx}
\usepackage{colortbl}

\makeatletter
\DeclareSymbolFont{mylargesymbols}{OMX}{ccex}{m}{n}
\DeclareMathDelimiter{\lbrace}{\mathopen}{symbols}{"66}{mylargesymbols}{"08}
\DeclareMathDelimiter{\rbrace}{\mathclose}{symbols}{"67}{mylargesymbols}{"09}
\DeclareMathDelimiter{(}{\mathopen}{operators}{"28}{mylargesymbols}{"00}
\DeclareMathDelimiter{)}{\mathclose}{operators}{"29}{mylargesymbols}{"01}
\DeclareMathDelimiter{[}{\mathopen}{operators}{"5B}{mylargesymbols}{"02}
\DeclareMathDelimiter{]}{\mathclose}{operators}{"5D}{mylargesymbols}{"03}
\DeclareMathSymbol{\braceld}{\mathord}{mylargesymbols}{"7A}
\DeclareMathSymbol{\bracerd}{\mathord}{mylargesymbols}{"7B}
\DeclareMathSymbol{\bracelu}{\mathord}{mylargesymbols}{"7C}
\DeclareMathSymbol{\braceru}{\mathord}{mylargesymbols}{"7D}
\makeatother

\makeatletter
\newcommand{\setword}[2]{%
  \phantomsection
  #1\def\@currentlabel{\unexpanded{#1}}\label{#2}%
}
\makeatother

\usepackage{amsmath}
\usepackage{amsthm,mathtools,amssymb}
\usepackage{blkarray}

\usepackage{mathrsfs} 

\DeclareFontFamily{U}{mathb}{\hyphenchar\font45}
\DeclareFontShape{U}{mathb}{m}{n}{
      <5> <6> <7> <8> <9> <10> gen * mathb
      <10.95> mathb10 <12> <14.4> <17.28> <20.74> <24.88> mathb12
      }{}
\DeclareSymbolFont{mathb}{U}{mathb}{m}{n}

\DeclareMathSymbol{\smalltriangleleft}{2}{mathb}{"9A}
\DeclareMathSymbol{\smalltriangleright}{2}{mathb}{"9B}

\newcommand{\gmmr}[2]{%
	\kern-0.7ex\mathrel{\raisebox{.19ex}{\scalebox{0.5}[0.5]{$#1$}}}%
	\kern-0.15ex\mathrel{\smalltriangleright\kern0.1ex\raisebox{.19ex}{\scalebox{0.5}[0.5]{$#2$}}}%
	\kern-0.6ex%
	}
\newcommand{\gmml}[2]{%
	\kern-0.6ex\mathrel{\raisebox{.19ex}{\scalebox{0.5}[0.5]{$#1$}}}%
	\kern0.1ex\mathrel{\smalltriangleleft\kern-0.15ex\raisebox{.19ex}{\scalebox{0.5}[0.5]{$#2$}}}%
	\kern-0.7ex%
	}
\DeclarePairedDelimiterX{\infdivx}[2]{(}{)}{%
  \kern1pt #1\delimsize\|\,#2%
}
\newcommand{\infdiv}{D_{\kern-0.5pt K\kern-1pt L\kern-0.5pt}\infdivx}

\renewcommand{\circle}{\kern-1pt\circ\kern-1pt}
\newtheorem{theorem}{Theorem}[section]
\newtheorem*{theorem*}{Implicit function theorem}
\newtheorem{corollary}{Corollary}[theorem]
\newtheorem{lemma}[theorem]{Lemma}

\theoremstyle{definition}

{
    \theoremstyle{plain}
    \newtheorem*{assumption1}{Assumption 1}
    \newtheorem*{assumption2}{Assumption 2} 
    \newtheorem*{assumption1'}{Assumption 1$^\prime$}
    \newtheorem*{assumption2'}{Assumption 2$^\prime$}
}

\theoremstyle{remark}
\newtheorem{remark}{\bf Remark}[section]
\newtheorem{example}{\bf Example}[section]

\usepackage{authblk}
\usepackage{hyperref}
\providecommand{\keywords}[1]{\textbf{\textit{Index terms---}} #1}

\DeclareMathOperator{\sign}{sgn}
\newcommand*{\Scale}[2][4]{\scalebox{#1}{$#2$}}%

\newcommand{\interior}[1]{%
  {\kern0pt#1}^{\mathrm{o}}%
}

\title{Analytic Implicit Functions}
\author[]{Kyung Soo Rim\thanks{This work was supported by a grant from the National Research Foundation of Korea, Grant No. NRF-2017R1E1A1A03070307.} }
\affil[]{Department of Mathematics, Sogang University}
\affil[]{\small\texttt{ksrim@sogang.ac.kr}}

\date{\today}

\begin{document}
\maketitle

\begin{abstract}
In this paper, we introduce a method of converting implicit equations to the usual forms of  functions locally without differentiability.
For a system of implicit equations which are equipped with continuous functions, 
if there are  
unique analytic implicit functions, that satisfies the system
in some rectangle, then each analytic function is represented as a power series which is the weak-star limit of
partial sums in the space of essentially bounded functions. 
%
We also provide numerical examples in order to demonstrate how the theoretical results in this article can be applied in practice and to show the effectiveness of the suggested approaches.
%
\end{abstract}

\section{Introduction}
Setting up implicit equations and solving them has long been so important that it has virtually come to describe what mathematical analysis and its applications are all about. 
Many useful mathematical models have expressions of implicit functions, even for problems of minimizing or maximizing functions subject to constraints.
From information of a single point solution, the implicit function theorem allows for understanding the relation between variables but, in spite of that, the implicit function theorems do not provide a usual function form defined explicitly. 
A central issue 
in this subject is how to derive a power series expansion from an equation involving parameters  because the representation as a function form enables us to estimate the mathematical models more easily, in economics, physics, engineering, etc., as well as mathematical development. 

In 1669, Issac Newton introduced one of the first instances of analyzing the behavior of an implicitly defined function originated from an equation jointed by mutually correlated variables 
(\cite{newton}). 
In 1684, Gottfried Leibniz applied implicit differentiation to calculate partial derivatives, 
which is a way to take the derivative of a term with respect to another variable without having to isolate either variable (\cite{struik}).
Joseph Lagrange, in 1770, derived an inversion formula which is one of the fundamental formulas of combinatorics. The formula is closely related to an inverse function theorem, by formulating a formal power series expansion for an implicit function \cite[Theorem 2.3.1]{kp}. In the 19th century, Augustin-Louis Cauchy was the first to state and solve an implicit equation using a rigorous mathematical form. He gave proof of the implicit function theorem having a form of power series by using 
Hardamard’s estimate for the Taylor coefficient of a given function,
which is induced formally
(\cite[Theorems 2.4.6 and 6.1.2]{kp}). Since then, there have been many improvements on the existence of usual function expressions
from implicit equations under suitable assumptions. 

Most importantly, in 1877, Ulisse Dini was the first to prove the real variable 
result of an implicit function theorem (\cite{dini}).
Conceived over two hundred years ago as a tool for studying mechanics and physics, the implicit function theorem has many formulations and is used in many aspects of mathematics. In addition, there are many particular types of implicit function theorems which are extended to Banach spaces, even under degenerate or non-smooth situations (e.g., \cite{nash,moser}). 
Implicit function theorems, even those that are quite sophisticated, are fundamental and powerful parts of the foundation of modern mathematics.
Almost all studies of an implicit function are related to its existence rather than showing how they behave. 

This research is devoted to determining a form of function of an implicit function which does not adopt the Taylor series of a function given, as used in Cauchy’s method, but instead of differentiability, we need the integrability of the given function.
This article highlights the formulations of 
a power series representation from an implicit equation
such that its partial sum converges weak-star in the space of essentially bounded functions. This method relies essentially on the continuity 
of a given function.
In addition, we present several examples for the computational validity of this study, along with a numerical calculation for each.
While there have been too many contributions to the theory of implicit functions to mention them all here, we would like to refer to books \cite{kp} and \cite{dr} as good overviews.


Throughout the article we use multi-index notations. 
Let $\alpha=(\alpha_1,\ldots,\alpha_n)$, $\beta=(\beta_1,\ldots,\beta_n)$ in $\mathbb{Z}^n$ and $c$ a scalar. 
We denote $\alpha\le\beta$, $\alpha\le c$, and $\alpha\pm c$ if $\alpha_k\le \beta_k$, $\alpha_k\le c$, and $(\alpha_1\pm c,\ldots,\alpha_n\pm c)$ for every $k$, respectively, where the plus–minus sign is replaced by either the plus or minus sign in the same order.
Moreover, 
for a vector-valued function $f$, $f_k$ denotes the $k$th component of $f$.
For indexed quantities $d_\alpha$, $\left(d_\alpha\right)_\alpha$ is called 
a tensor (or a matrix only when $\alpha$ has two components) and 
the determinant of a matrix $\left(d_\alpha\right)_\alpha$
is defined by $|\left(d_\alpha\right)_\alpha|$ 
or $\det \left(d_\alpha\right)_\alpha$.
The subsets $R$ and $I$ of $\mathbb{R}^n$ indicate rectangles
which are the Cartesian products of $n$ compact intervals and $|R|$ defines the Lebesgue measure of $R$.
Especially, $I$ represents the codomain of an implicit function.
Sometimes the integer $N$ is considered a multi-index, in which case all components are defined as $N$.
%
With these the main results are formulated in Theorems \ref{polynomial solution}, \ref{analytic solution}, and \ref{system of analytic solutions}.

Before understanding the proposed methods, we must first state the classical implicit function theorem as follows.

\begin{theorem*}
Let $f(x,y):\mathbb{R}^{n+m}\to\mathbb{R}^m$ be continuously differentiable with $f(a,b)=\mathbf{0}$.  
Suppose that $|J_{f,y}(a,b)|=\det\big(\partial_{y_j} f_i(a,b)\big)\ne0$. Then there is  
an $R\times I$ of $(a,b)$ and a unique continuously differentiable function $g:R\to I$ such that $f(x,g(x))=\mathbf{0}$ in $R$.
Moreover, 
the partial derivative of $g$ is given by 
$\partial_{x_j}g(x)=-J_{f,y}(\partial_{y_j}f_i(x,g(x)))_{m\times m}^{-1}(\partial_{x_j}f(x,g(x)))_{m\times1}$.
\end{theorem*}

This implicit function theorem has been extended to that of analytic functions. If every $f_k$ is analytic in the implicit function theorem, then every $g_k$ is also analytic (refer to \cite[Theorem 6.1.2]{kp}).

\section{Implicit functions} \label{implicit functions}

In this section, we derive 
the integral representation of an implicit function to isolate a dependent variable
even though it is neither transcendental nor algebraic.
Let $\sign$ be the signum function on $\mathbb{R}$ and $f(x,y)$ a function on $\mathbb{R}^{n+1}$. 
For each $x$, define a function $\sign_y f(x,y)$ by $1$ if $f(x,y)\ge0$ and $-1$ otherwise.
We begin with a fundamental condition for the local existence of an implicit function.

\begin{assumption1}
Let $f:\mathbb{R}^{n+1}\to\mathbb{R}$ be continuous. 
Then there is an $R\times I\subset\mathbb{R}^{n+1}$ 
such that for each $x\in R$, $\sign_y f(x,y)$
has only one jump discontinuity on $I$.
\end{assumption1}

From Assumption 1, by the intermediate value theorem, by the continuity of $f$, and by the uniqueness of jump discontinuity, there is a continuous function $g:R\to I$ such that $f(x,g(x))=0$ in $R$. 
Moreover, if $f$ is continuously differentiable such that $\partial_yf(a,b)\ne0$ with $f(a,b)=0$, then
by the implicit function theorem 
there is an $R\times I\ni(a,b)$,  which has a unique continuously differentiable function $y=g(x)$ such that $f(x,g(x))=0$ in $R$. 
Also, 
the Jacobian matrix is invertible in $R$. This implies that  
for each $x$, $\sign_y f(x,y)$ must have only one jump discontinuity on $I$.
Thus, the sufficient condition of the implicit function theorem guarantees Assumption 1.

\begin{example}
\begin{itemize}
\item[$(i)$]
Let $f(x,y)=x^2+y^2-1$ be a function with $f(0,1)=0$. We want to find a function for $y$ such that $f=0$ in some closed interval of $x=0$.
Let us consider $R\times I=[-1,1]\times[0,2]\ni(0,1)$. Then   
$\sign_y f(x,y)$ has only one jump discontinuity on $I$. 
(On the other hand,
since $\partial_yf(0,1)=2\ne0$, by application of the implicit function theorem
there is an $R\times I$ of $(0,1)$ such that for each $x\in R$, $\sign_yf(x,y)$ has the only one jump discontinuity on $I$.) 
%

%
\item[$(ii)$]
Let $f(x,y)=y^2-x^4$ be a function with $f(0,0)=0$. Although $\partial_y f(0,0)=0$, 
by taking $R\times I=[-1,1]\times[0,1]$ or $[-1,1]\times[-1,0]\ni(0,0)$, we can justify that $f$ satisfies Assumption 1. 
\end{itemize}
\end{example}

For each $x\in R$, if we define $n_y(f)(x)=1$ if $\sign_y f(x,y)$ is increasing and $-1$ if $\sign_y f(x,y)$ is decreasing,
%
then, $n_y(f)$ is constant. 
Indeed, by the intermediate value theorem
we may let $y(x)\in I$ be a unique solution such that $f(x,y(x))=0$ for each $x\in R$.
Suppose that there is $\xi_1$ and $\eta_1$ such that $n_y(f)(\xi_1)n_y(f)(\eta_1)=-1$. We may assume that $n_y(f)(\xi_1)=1$ and $n_y(f)(\eta_1)=-1$. 
We choose $r>0$ such that $y(x)+r\in I$ or $y(x)-r\in I$
for every $x\in R$.
Let $x_0$ be the midpoint of the line segment $\overline{\xi_1\eta_1}$ between $\xi_1$ and $\eta_1$.
If $f(x_0,y(x_0)+r)>0$ or $f(x_0,y(x_0)-r)<0$, then select $\overline{x_0\eta_1}$, whereas if $f(x_0,y(x_0)+r)<0$ or $f(x_0,y(x_0)-r)>0$, then select $\overline{\xi_1x_0}$. 
Write the selected one as $\overline{\xi_2\eta_2}$. 
This procedure reaches the same situation 
of $n_y(f)(\xi_2)=1$ and $n_y(f)(\eta_2)=-1$
as the beginning.
Repeating the above process inductively by taking the midpoint of $\overline{\xi_2\eta_2}$, we have the monotone sequences $\xi_k$ and $\eta_k$, so that $\xi_k\nearrow\zeta$ and $\eta_k\searrow\zeta$ for some $\zeta\in R$. The convergence of $\xi_k$ and $\eta_k$ and the continuity of $f$ yield the convergence of $y(\xi_k)$ and $y(\eta_k)$, i.e., $y(\xi_k)$ and $y(\eta_k)$ go to $y(\zeta)$ as $k\to\infty$. 

It follows that
$f(\xi_k,y(\xi_k)+r)>0$ (or $f(\xi_k,y(\xi_k)-r)<0$)
and
$f(\eta_k,y(\eta_k)+r)<0$ (or $f(\eta_k,y(\eta_k)-r)>0$) for every $k$.
Taking the limit $k\to\infty$, we have
$f(\zeta,y(\zeta)+r)>0$ (or $f(\zeta,y(\zeta)-r)<0$)
and
$f(\zeta,y(\zeta)+r)<0$ (or $f(\zeta,y(\zeta)-r)>0$). This leads contradiction.
Therefore, we conclude the following lemma.

\begin{lemma} \label{fixing}
In $R$, $n_y(f)$ is constant (either $1$ or $-1$).
\end{lemma}

Let $\Theta$ be the Heaviside step function.
For each $x$, define a function $\Theta_y f(x,y)$ that assigns $1$ if $f(x,y)\ge 0$ and $0$ otherwise.
For a rectangle $R'\subset R$,
we define a quantity as 
\begin{equation} \label{volume integration in poly}
	\iint_{R'\times I} \Theta_y f(x,y)\,dx\,dy = v_{R'\times I}.
\end{equation}
The next result provides a necessary condition for the existence of a function form of an implicit equation. 

\begin{lemma} \label{sign} 
If $R\times I$ contains a unique function $y=g(x)$ such that $f(x,y)=0$, then  
\begin{equation} \label{summary of volume element}
	v_{R'\times I} = \frac{n_y(f)+1}{2}|R'|\max I 
		+ \frac{n_y(f)-1}{2}|R'| \min I  - n_y(f)\int_{R'}g\,dx
\end{equation}
for every rectangle $R'\subset R$.
\end{lemma}

By Lemma \ref{fixing}, the right-hand side of (\ref{summary of volume element}) is well defined.

\begin{proof}[Proof of Lemma \ref{sign}]
Let $R'\subset R$ be a rectangle.
Suppose that $n_y(f)=1$. 
The integration (\ref{volume integration in poly}) is calculated as
\begin{equation} \label{sign 1}
	\iint_{R'\times I} \Theta_y f(x,y)\,dx\,dy = |R'|\max I - \int_{R'}g\,dx.
\end{equation}
If $n_y(f)=-1$, then
by putting $F(x,y)=f(x,-y)$, we have $n_y(F)=1$ on $-I$ with $y=-g(x)$ which satisfies $F(x,y)=0$. By (\ref{sign 1}) with $F$ and $y=-g(x)$ on $R\times(-I)$ instead of $f$ and $y=g(x)$,
\begin{equation} \label{sign 1-1}
\iint_{R'\times(-I)}\Theta_yF(x,y)\,dx\,dy=|R'|\max (-I) +\int_{R'}g\,dx.
\end{equation}
By the change of variables ($-y\mapsto y$), (\ref{sign 1-1}) equals
\begin{equation}\label{criterion1-3}
\iint_{R'\times I}\Theta_yf(x,y)\,dx\,dy=-|R'|\min(I)+\int_{R'}g\,dx.
\end{equation}
According to (\ref{sign 1}) and (\ref{criterion1-3}), we conclude 
(\ref{summary of volume element}). Therefore, the proof is complete. 
\end{proof}

Now, we have the following integral representation of an implicit function. 

\begin{theorem} \label{coro of representation}
The continuous implicit function $y=g(x)$ such that $f=0$ in $R$, is given by  
\begin{equation}\label{int repre}
g(x) = \frac{1+n_y(f)}{2}\max I + \frac{1-n_y(f)}{2}\min I -n_y(f)\int_I\Theta_y f(x,y)\,dy
\qquad(x\in R).
\end{equation}
\end{theorem}

Equation (\ref{int repre}) does not depend on the choice of a rectangle $R\times I$
whenever it contains $y=g(x)$.

\begin{proof}[Proof of Theorem \ref{coro of representation}]
Let $x\in R$ and take cubes $Q\subset R$ that shrink to a singleton $\{x\}$. 
Divide $|Q|$ into both sides of (\ref{volume integration in poly}) and (\ref{summary of volume element}) after replacing $R'$ with $Q$.
By Fubini's theorem and by the Lebesgue differentiation theorem,
\begin{equation*} 
	\int_I \Theta_y f(x,y)\,dy = \frac{n_y(f)+1}{2}\max I 
		+ \frac{n_y(f)-1}{2}\min I  - n_y(f)g(x)
\end{equation*}
for every point of continuity of $f$.
Therefore, the proof is complete.
\end{proof}

Now, we set up an algebraic operator that will appear in the main theorems. 
Let $C=(c_{i_1\cdots i_N})$ and $A=(a_{j_1j_2})$ be an $m_1\times \cdots \times m_N$ tensor and $n_1\times n_2$ matrix, respectively.
If $n_1=m_k$, then we
define a tensor contraction between $A$ and $C$ by
\begin{equation*}
	A\stackrel{1\to k}{\cdot} C 
		= \left(\sum_{i_k=1}^{m_k}a_{i_k j_2}c_{\cdots i_k\cdots}\right)
\end{equation*}
which produces an $m_1\times \cdots \times m_{k-1}\times n_2 \times m_{k+1}\times \cdots \times m_N$ tensor.
If $n_2=m_k$, then $A\stackrel{2\to k}{\cdot} C$ is also defined, and it becomes an $m_1\times \cdots \times m_{k-1}\times n_1 \times m_{k+1}\times \cdots \times m_N$ tensor.
For another $n_1'\times n_2'$ matrix $B$,
if $n_1=m_i$ and $n_2'=m_j$ $(i\ne j)$, then we have the commutative property of 
\begin{equation*}
	A\stackrel{1\to i}{\cdot} (B\stackrel{2\to j}{\cdot} C)
		= B\stackrel{2\to j}{\cdot}(A\stackrel{1\to i}{\cdot}C).
\end{equation*}
Here, the parentheses are omitted providing the commutative property.

\section{Polynomial implicit functions} \label{pif}

In this section, we derive the implicit function of $y=g(x)$ such that $f(x,y)=0$, provided 
$y=g(x)$ is   
a unique multivariate polynomial on $R\times I$.
Put $R=\prod_{k=1}^{n}[\xi_k,\eta_k]$.
Let  $N_k$ be a partition number and 
$\xi_k<\xi_k+\Delta_k< \xi_k+2\Delta_k<\cdots< \xi_k+N_k\Delta_k=\eta_k$
a partition  on $[\xi_k,\eta_k]$,
where $\Delta_k=(\eta_k-\xi_k)/N_k$ is a partition size.
We define a grid block $R_\alpha$ by
\begin{equation}\label{block decomposition}
	R_{\alpha}=\prod_{k=1}^n 
	\left[\xi_{k}+(\alpha_k-1)\Delta_k, 
	\xi_{k}+\alpha_k\Delta_k\right]
	\quad(1\le \alpha_k\le N_k),
\end{equation}
where $1\le\alpha_k\le N_k$ $(1\le k\le n)$.
The union of all $R_{\alpha}$ is  $R$ and 
$|R_{\alpha}|=\prod_{k=1}^{n}(\eta_k-\xi_k)/N_k$.
For $a\in \mathbb{R}^n$,
we define a matrix by
\begin{equation} \label{vander}
	V_{N_k} = \begin{pmatrix}
				\left(\xi_{k}+\alpha_k\Delta_k-a_k\right)^j 
				- \left(\xi_{k}+(\alpha_k-1)\Delta_k-a_k\right)^j
			 \end{pmatrix}_{N_k\times N_k}
			 \quad(1\le \alpha_k,\,j\le N_k),
\end{equation}
where   
$\alpha_k$ and $j$ denote a row and column number, respectively.

\begin{lemma}\label{invertibility} 
For every $N_k$ $(1\le k\le n)$, $|V_{N_k}|$ is positive and independent of $a$. Consequently,   
$V_{N_k}$ is invertible.
\end{lemma}

\begin{proof}
Fix $k$ $(1\le k\le n)$. Writing $\gamma_k=\xi_k-a_k$, we get
\begin{equation} \label{determinant1}
	|V_{N_k}|
	= 
	\left|
	\begin{matrix}
	\left(\left(\gamma_{k}+\alpha_k\Delta_k\right)^j 
		- \left(\gamma_{k}+(\alpha_k-1)\Delta_k\right)^j\right)_{(k,j)}
	\end{matrix}\right|.  
\end{equation}
From the invariant properties of determinants, 
add the first row to the second, the second row to the third,
and so on until the end. Then
(\ref{determinant1}) equals
\begin{equation} \label{determinant}
	\arraycolsep=3pt\def\arraystretch{1.5}
	\left|\;
	\begin{matrix} 
	\Delta_k & \left(\gamma_k+\Delta_k\right)^2 - \gamma_k^2 & \cdots 
		& \left(\gamma_k+\Delta_k\right)^{N_k}- \gamma_k^{N_k} \\
	2\Delta_k & \left(\gamma_k+2\Delta_k\right)^2 - \gamma_k^2 & \cdots 
		& \left(\gamma_k+2\Delta_k\right)^{N_k} - \gamma_k^{N_k} \\
	\vdots & \vdots & \ddots & \vdots  \\
	 N_k\Delta_k & \left(\gamma_k+N_k\Delta_k\right)^2 - \gamma_k^2 & \cdots 
	 	& \left(\gamma_k+N_k\Delta_k\right)^{N_k} - \gamma_k^{N_k}\\
	\end{matrix}\right|  
\end{equation}
which is also equal to
\begin{equation} \label{determinant2}
	\arraycolsep=3pt\def\arraystretch{1.5}
	\left|\;
	\begin{matrix} 
	1  &  \gamma_k &  \gamma_k^2 & \cdots 
		&  \gamma_k^{N_k}  \\
	0  &  \Delta_k & \left(\gamma_k+\Delta_k\right)^2 - \gamma_k^2 & \cdots 
		& \left(\gamma_k+\Delta_k\right)^{N_k} - \gamma_k^{N_k} \\
	0  &  2\Delta_k & \left(\gamma_k+2\Delta_k\right)^2 - \gamma_k^2 & \cdots 
		& \left(\gamma_k+2\Delta_k\right)^{N_k} - \gamma_k^{N_k} \\
	\vdots  & \vdots & \vdots & \ddots & \vdots  \\
	0  &  N_k\Delta_k & \left(\gamma_k+N_k\Delta_k\right)^2 - \gamma_k^2 & \cdots 
		& \left(\gamma_k+N_k\Delta_k\right)^{N_k} - \gamma_k^{N_k} \\
	\end{matrix}\right| 
\end{equation}
by appending the first row and column of (\ref{determinant2}) to (\ref{determinant}).
By adding the first row of (\ref{determinant2}) to all other rows, we have
\begin{equation} \label{determinant3}
	\mbox{(\ref{determinant2})}= 
	\arraycolsep=3pt\def\arraystretch{1.5}
	\left|\;
	\begin{matrix} 
	1  &  \gamma_k &  \gamma_k^2 & \cdots 
		&  \gamma_k^{N_k}  \\
	1  &  \gamma_k+\Delta_k & \left(\gamma_k+\Delta_k\right)^2 & \cdots 
		& \left(\gamma_k+\Delta_k\right)^{N_k} \\
	1  &  \gamma_k+2\Delta_k & \left(\gamma_k+2\Delta_k\right)^2 & \cdots 
		& \left(\gamma_k+2\Delta_k\right)^{N_k} \\
	\vdots  & \vdots & \vdots & \ddots & \vdots  \\
	1  &  \gamma_k+N_k\Delta_k & \left(\gamma_k+N_k\Delta_k\right)^2 & \cdots 
		& \left(\gamma_k+N_k\Delta_k\right)^{N_k} \\
	\end{matrix}\right| 
\end{equation}
which is  
\begin{equation*} \label{determinant3-1}
 \Delta_k\prod_{0\le i<j\le N_k} (j-i) >0,
\end{equation*}
where the quantity is strictly positive for any $N_k$ and independent of $a$.  
Therefore, the proof is complete. 
\end{proof}

For notational simplicity, write $v_{R_{\alpha}\times I}=v_{\alpha}$ 
in (\ref{volume integration in poly})  and put
\begin{equation} \label{data}
	d_{\alpha}= \frac{1+n_y(f)}{2}|R_{\alpha}|\max I 
				+ \frac{1-n_y(f)}{2}|R_{\alpha}|\min I 
				- n_y(f) v_{\alpha}
\end{equation}
which comes only from $f$.	
By Lemma \ref{sign}, $d_\alpha$ can also be calculated as
\begin{equation} \label{2nd data}
	d_{\alpha}=\int_{R_{\alpha}}g\,dx.
\end{equation}	

Let $W$ be an $N_1\times\cdots\times N_n$ tensor 
whose $\alpha$th component is $\alpha_1\cdots\alpha_n$.
From now on, set $N=(N_1,\ldots,N_n)$.

%

\begin{theorem} \label{polynomial solution}
If $y=g(x)$ is a multivariate polynomial from $R$ to $I$ such that $f(x,y)=0$,  then 
for $a\in R$, $g(x)=\sum_{0\le \beta< N}c_{\beta}(x-a)^{\beta}$, where 
$c_{\beta}$ is the $(\beta+1)$th component of   
\begin{equation} \label{1st coefficients in poly}
	W\circ\Big[V_{N_n}^{-1}\stackrel{2\to n}{\cdot}
	\cdots
	V_{N_2}^{-1}\stackrel{2\to 2}{\cdot}
	V_{N_1}^{-1}\stackrel{2\to 1}{\cdot}
	\begin{pmatrix}d_{\alpha}\end{pmatrix}_\alpha\Big]\qquad(1\le\alpha\le N),
\end{equation}
$N_k>$ the largest exponent of $x_k$, and
$\circ$ denotes the Hadamard product between two tensors. 
\end{theorem}

In Theorem \ref{polynomial solution}, 
every component of (\ref{1st coefficients in poly}) vanishes if its index contains a number larger than the largest exponent of $x_k$ for some $k$.
Now we will call (\ref{1st coefficients in poly}) the coefficient tensor for $g$.

\begin{proof}[Proof of Theorem \ref{polynomial solution}]
For each $k$, take a sufficiently large $N_k$ so that $N_k>$ the largest exponent of $x_k$ in $g$.
By (\ref{block decomposition}), 
\begin{equation} \label{evaluation in poly}
\begin{split} 
	\int_{R_{\alpha}} g\,dx
		&= \sum_{0\le\beta< N}c_{\beta} \int_{R_{\alpha}}(x-a)^{\beta}\,dx \\
		&= \sum_{0\le\beta< N} 
			c_{\beta}\prod_{k=1}^n
			\frac{1}{\beta_k+1}
			\left[
			\left(\xi_{k}+\alpha_k\Delta_k-a_k\right)^{\beta_k+1} 
			- \left(\xi_{k}+(\alpha_k-1)\Delta_k-a_k\right)^{\beta_k+1} 
			\right].
\end{split}
\end{equation}	
By (\ref{2nd data}), the sum of (\ref{evaluation in poly}) is equal to $d_\alpha$.
Moreover, (\ref{evaluation in poly}) is the $\alpha$th component of 
\begin{equation} \label{evaluation poly in matrix}
	V_{N_1}\stackrel{2\to 1}{\cdot}
	V_{N_2}\stackrel{2\to 2}{\cdot}
	\cdots 
	V_{N_n}\stackrel{2\to n}{\cdot} 
	\begin{pmatrix}c_{\beta}/(\beta_1+1)\cdots(\beta_n+1)\end{pmatrix}_\beta,
\end{equation}
where $0\le\beta\le N-1$. Since $V_{N_k}$ is invertible by Lemma \ref{invertibility}, we have
\begin{equation*} \label{solution coeff1-1}
	\begin{pmatrix}c_{\beta}/(\beta_1+1)\cdots(\beta_n+1)\end{pmatrix}_\beta
	=
	V_{N_n}^{-1}\stackrel{2\to n}{\cdot}
	\cdots
	V_{N_2}^{-1}\stackrel{2\to 2}{\cdot}
	V_{N_1}^{-1}\stackrel{2\to 1}{\cdot}
	\begin{pmatrix}d_{\alpha}\end{pmatrix}_\alpha.
\end{equation*}
So,
\begin{equation} \label{solution coeff1}
	\begin{pmatrix}c_{\beta}\end{pmatrix}_\beta
	=
	W\circ \Big[V_{N_n}^{-1}\stackrel{2\to n}{\cdot}
	\cdots
	V_{N_2}^{-1}\stackrel{2\to 2}{\cdot}
	V_{N_1}^{-1}\stackrel{2\to 1}{\cdot}
	\begin{pmatrix}d_{\alpha}\end{pmatrix}_\alpha\Big].
\end{equation}

For the $N'$-partition of $R$ $(N'\ge N)$, according to (\ref{block decomposition}), (\ref{vander}), (\ref{data}), (\ref{2nd data}), and (\ref{solution coeff1}), let $\big(c_{\beta}'\big)_{\beta}$ be
the coefficient tensor for $g(x)=\sum_{0\le\beta<N'}c_\beta'(x-a)^\beta$ which satisfies $f=0$.  
From the uniqueness of the implicit function, the nontrivial components of two tensors 
should be identical, but 
they are different only in their sizes
by adding zero components.
This implies that $c_\beta$ does not depend on the choice of $N$ whenever $N_k$ is greater than the largest exponent of $x_k$ of $g(x)$.
Therefore, (\ref{solution coeff1}) results in the desired conclusion (\ref{1st coefficients in poly}).
\end{proof}

Since the necessary condition of the implicit function theorem with $m=1$ implies Assumption 1, 
as mentioned before, we have a corollary. 

\begin{corollary} \label{poly ift}
Suppose that $f$ is continuously differentiable such that $\partial_yf(a,b)\ne0$ with $f(a,b)=0$.
If $y=g(x)$ is a multivariate polynomial such that $f(x,y)=0$ near $(a,b)$, then there is an $R\times I$ so that for $a\in R$, $g(x)=\sum_{0\le \alpha< N}c_{\alpha}(x-a)^{\alpha}$ with the 
coefficient tensor of (\ref{1st coefficients in poly}).
\end{corollary}

The following example is about a polynomial implicit function with two independent variables.

\begin{example}
Let $f:\mathbb{R}^3\to\mathbb{R}$ be given by
\begin{equation*}
\begin{split}
	f(x,y,z) &= 0.5x^4 + 0.5 x^3 y + 0.5 x^3 + 2 x^2 y + 0.5 x^2 z + 0.5 x y^2 \\
		   &\qquad	- 0.5 x y z + 1.5 x y - 0.5 x z + x + 1.5 y^2 - 0.5 y z + 2 y - z^2 + 3 z - 2
\end{split}		   
\end{equation*}
with $(2,0,-2)$ and $(1,1,4)$ at which $f$ vanishes. 
We want to solve $f=0$ for $z$ 
as a function of $x$ and $y$.

First, since $\partial_zf(2,0,-2)=8\ne0$,
by application of the implicit function theorem there is a rectangle of $(2,0,-2)$ on which
$\sign_yf(x,y)$ has the only one jump discontinuity.
Choose $R=[1.5,2.5]\times[-0.5,0.5] \ni (2,0)$ and $I=[-5,1] \ni -2$,
so that $n_z(f)=1$ in $R\times I$.
The surface of $f=0$ and $R\times I$ are  depicted in Figure \ref{fig1}.
With $N_1=N_2=3$, $V_{N_1}$, $V_{N_2}$ and $W$ are calculated as
\begin{equation*}
\arraycolsep=3.5pt\def\arraystretch{1.2}
		\begin{pmatrix}
		\Scale[0.8]{\Delta_x} & \Scale[0.8]{(-0.5+\Delta_x)^2-(-0.5)^2} 
			& \Scale[0.8]{(-0.5+\Delta_x)^3- (-0.5)^3} \\
		\Scale[0.8]{\Delta_x} & \Scale[0.8]{(-0.5+2\Delta_x)^2-(-0.5+\Delta_x)^2} 
			& \Scale[0.8]{(-0.5+2\Delta_x)^3- (-0.5+\Delta_x)^3} \\
		\Scale[0.8]{\Delta_x} & \Scale[0.8]{(-0.5+3\Delta_x)^2-(-0.5+2\Delta_x)^2} 
			& \Scale[0.8]{(-0.5+3\Delta_x)^3- (-0.5+2\Delta_x)^3} 
		\end{pmatrix}, 	
\end{equation*}
\begin{equation*}
\arraycolsep=3.5pt\def\arraystretch{1.2}
		\begin{pmatrix}
		\Scale[0.8]{\Delta_y} & \Scale[0.8]{(-0.5+\Delta_y)^2-(-0.5)^2} 
			& \Scale[0.8]{(-0.5+\Delta_y)^3- (-0.5)^3} \\
		\Scale[0.8]{\Delta_y} & \Scale[0.8]{(-0.5+2\Delta_y)^2-(-0.5+\Delta_y)^2} 
			& \Scale[0.8]{(-0.5+2\Delta_y)^3- (-0.5+\Delta_y)^3} \\
		\Scale[0.8]{\Delta_y} & \Scale[0.8]{(-0.5+3\Delta_y)^2-(-0.5+2\Delta_y)^2} 
			& \Scale[0.8]{(-0.5+3\Delta_y)^3- (-0.5+2\Delta_y)^3} 
		\end{pmatrix}, 	
\end{equation*}
and
\begin{equation*}
\begin{pmatrix}
		\Scale[0.8]{1\cdot 1} & \Scale[0.8]{1\cdot 2}  & \Scale[0.8]{1\cdot 3} \\
		\Scale[0.8]{2\cdot 1} & \Scale[0.8]{2\cdot 2}  & \Scale[0.8]{2\cdot 3} \\
		\Scale[0.8]{3\cdot 1} & \Scale[0.8]{3\cdot 2}  & \Scale[0.8]{3\cdot 3} \\
		\end{pmatrix},	
\end{equation*}
respectively. By  (\ref{data}), the matrix of $d_{(i,j)}$ is given by
\begin{equation*}
	\arraycolsep=2pt\def\arraystretch{1}
	\begin{pmatrix}
	\Scale[0.8]{-0.049896967860677}	& \Scale[0.8]{-0.136316792736493} & \Scale[0.8]{-0.222736578394150} \\
	\Scale[0.8]{-0.130143979013675} & \Scale[0.8]{-0.222736592404210} &  \Scale[0.8]{-0.315329196303353} \\
	\Scale[0.8]{-0.222736603139157} & \Scale[0.8]{-0.321502042403124} &  \Scale[0.8]{-0.420267478899573}
	\end{pmatrix}.
\end{equation*}
By Corollary \ref{poly ift}, the desired function $z=g(x,y)=\sum_{0\le\alpha<3}c_{\alpha}(x-2)^{\alpha_1}y^{\alpha_2}$ is determined by the coefficient matrix of
\begin{equation}  \label{approx1}
W\circ \Big[V_{N_2}^{-1}\stackrel{2\to 2}{\cdot}V_{N_1}^{-1} \stackrel{2\to 1}{\cdot}
		\begin{pmatrix}d_{\alpha}\end{pmatrix}_{\alpha} \Big] \\
=
\arraycolsep=2.pt\def\arraystretch{1.2}
\begin{matrix}
& \\
\left(\hspace{0.5em} \vphantom{ \begin{matrix} 12 \\ 12 \\ 12 \end{matrix} } \right .
\end{matrix} 
\hspace{-1.2em} 
\begin{matrix}
\Scale[0.8]{1} & \Scale[0.8]{y} & \Scale[0.8]{y^2}  \\ \hline \\[-4mm]
\Scale[0.8]{-2} & \Scale[0.8]{-2.5} &  \Scale[0.8]{0} \\
\Scale[0.8]{-2.500001} & \boxed{\Scale[0.8]{-0.499998}} & \Scale[0.8]{-2\times 10^{-6}} \\
\Scale[0.8]{-0.499999} & \Scale[0.8]{-3\times10^{-6}} & \Scale[0.8]{4\times 10^{-6}} 
\end{matrix}
\hspace{-0.2em}
\begin{matrix}
& \\
\left . \vphantom{ \begin{matrix} 12 \\ 12 \\ 12 \end{matrix} } \right )
\begin{matrix} 
\Scale[0.8]{1} \\ \Scale[0.8]{x-2}  \\ \Scale[0.8]{(x-2)^2} 
\end{matrix}
\end{matrix}  \\,
\end{equation}
where $-0.499998$, for example, denotes the coefficient of $(x-2)y$ in the summation of $g$.

On the other hand, since $\partial_zf(1,1,4)= -6\ne0$, 
by application of the implicit function theorem there is a rectangle, for example,  
$R=[0.5, 1.5]\times [0.5, 1.5]\ni(1,1)$ and $I=[2, 7]\ni4$, on which 
$n_z(f)=-1$.
For $N_1=N_2=4$ (as shown in Figure \ref{fig1}), we derive the coefficient matrix of
\begin{equation}  \label{approx2}
\arraycolsep=2.pt\def\arraystretch{1.0}
\begin{matrix}
& \\
\left(\hspace{0.5em} \vphantom{ \begin{matrix} 12 \\ 12 \\ 12 \\ 12 \end{matrix} } \right .
\end{matrix}
\hspace{-0.9em}
\begin{matrix}
\Scale[0.8]{1} & \Scale[0.8]{y-1} & \Scale[0.8]{(y-1)^2}  & \Scale[0.8]{(y-1)^3} \\ \arrayrulecolor{gray}\hline \\[-4mm]
\Scale[0.8]{4} & \Scale[0.8]{1.000001} & \Scale[0.8]{-1\cdot 10^{-6}} & \Scale[0.8]{1\cdot 10^{-6}} \\
\Scale[0.8]{2.000001} & \Scale[0.8]{-1\cdot 10^{-6}} & \Scale[0.8]{2\cdot 10^{-6}} & \Scale[0.8]{-2\cdot 10^{-6}} \\
\Scale[0.8]{0.999999} & \Scale[0.8]{2\cdot 10^{-6}} & \Scale[0.8]{-3\cdot 10^{-6}} & \Scale[0.8]{4\cdot 10^{-6}} \\
\Scale[0.8]{0} & \Scale[0.8]{-1\cdot 10^{-6}} & \Scale[0.8]{3\cdot 10^{-6}} & \Scale[0.8]{-4\cdot 10^{-6}}
\end{matrix}
\hspace{-0.0em}
\begin{matrix}
& \\
\left . \vphantom{ \begin{matrix} 12 \\ 12 \\ 12 \\ 12 \end{matrix} } \right )
\begin{matrix}
\Scale[0.8]{1} \\ \Scale[0.8]{x-1}  \\ \Scale[0.8]{(x-1)^2} \\ \Scale[0.8]{(x-1)^3} 
\end{matrix}
\end{matrix}  
\end{equation}
of $z=g(x,y)=\sum_{0\le\alpha<4}c_{\alpha}(x-1)^{\alpha_1}(y-1)^{\alpha_2}$ that satisfies $f=0$ in $R$ by the same method shown above.

In fact, $f$ is factorized by 
\begin{equation*}
f_1 = \tfrac52(x-2)+\tfrac52 y +z +2+\tfrac12(x-2)^2 +\tfrac12(x-2)y
\end{equation*}
and
\begin{equation*}
f_2 = 2(x-1)+(y-1) -z +4 +(x-1)^2.
\end{equation*}
The coefficient matrices of functions for $z$ such that $f_1=0$ and $f_2=0$, respectively, are shown as 
\begin{equation*}
\arraycolsep=10.pt\def\arraystretch{1.2}
\begin{matrix}
& \\
\left(\hspace{-0.2em} \vphantom{ \begin{matrix} 12 \\ 12 \\ 12 \end{matrix} } \right .
\end{matrix} 
\hspace{-1.2em} 
\begin{matrix}
\Scale[0.8]{1} & \Scale[0.8]{y} & \Scale[0.8]{y^2}  \\ \hline \\[-4mm]
\Scale[0.8]{-2} & \Scale[0.8]{-2.5} &  \Scale[0.8]{0} \\
\Scale[0.8]{-2.5} & \Scale[0.8]{-0.5} & \Scale[0.8]{0} \\
\Scale[0.8]{-0.5} & \Scale[0.8]{0} & \Scale[0.8]{0} 
\end{matrix}
\hspace{0.2em}
\begin{matrix}
& \\
\left . \vphantom{ \begin{matrix} 12 \\ 12 \\ 12 \end{matrix} } \right )
\begin{matrix} 
\Scale[0.8]{1} \\ \Scale[0.8]{x-2}  \\ \Scale[0.8]{(x-2)^2} 
\end{matrix}
\end{matrix}
\end{equation*}
and
\begin{equation}
\arraycolsep=6.pt\def\arraystretch{1.0}
\begin{matrix}
& \\
\left(\hspace{0.5em} \vphantom{ \begin{matrix} 12 \\ 12 \\ 12 \\ 12 \end{matrix} } \right .
\end{matrix}
\hspace{-0.9em}
\begin{matrix}
\Scale[0.8]{1} & \Scale[0.8]{y-1} & \Scale[0.8]{(y-1)^2}  & \Scale[0.8]{(y-1)^3} \\ \arrayrulecolor{gray}\hline \\[-4mm]
\Scale[0.8]{4} & \Scale[0.8]{1} & \Scale[0.8]{0} & \Scale[0.8]{0} \\
\Scale[0.8]{2} & \Scale[0.8]{0} & \Scale[0.8]{0} & \Scale[0.8]{0} \\
\Scale[0.8]{1} & \Scale[0.8]{0} & \Scale[0.8]{0} & \Scale[0.8]{0} \\
\Scale[0.8]{0} & \Scale[0.8]{0} & \Scale[0.8]{0} & \Scale[0.8]{0}
\end{matrix}
\hspace{-0.0em}
\begin{matrix}
& \\
\left . \vphantom{ \begin{matrix} 12 \\ 12 \\ 12 \\ 12 \end{matrix} } \right )
\begin{matrix}
\Scale[0.8]{1} \\ \Scale[0.8]{x-1}  \\ \Scale[0.8]{(x-1)^2} \\ \Scale[0.8]{(x-1)^3} 
\end{matrix}
\end{matrix}  
\end{equation}
which are comparable to (\ref{approx1}) and
(\ref{approx2}), respectively.
\begin{figure}[!ht]
\centerline{\includegraphics[clip, trim=0cm 8cm 0cm 8cm,width=0.70\textwidth]{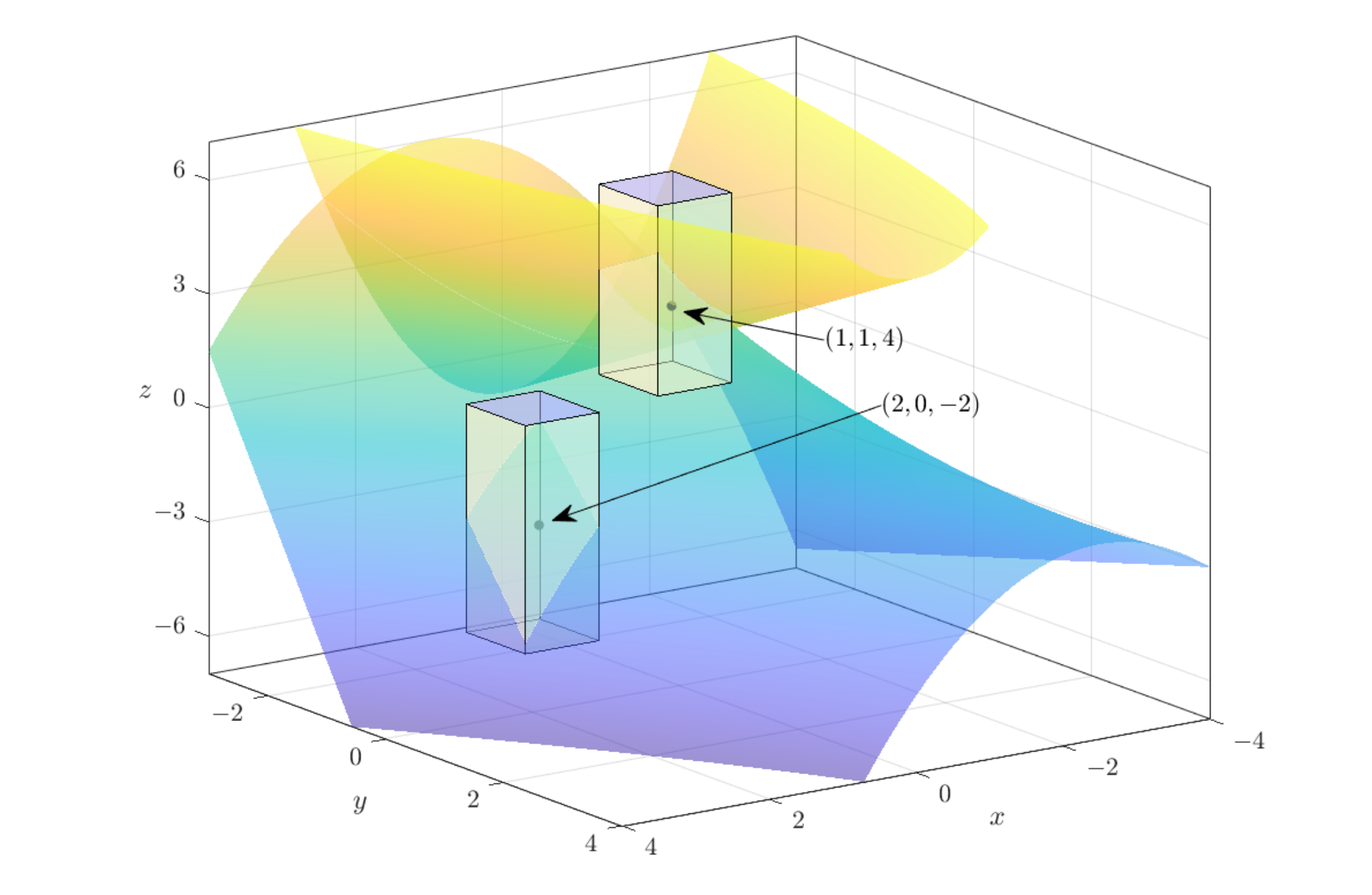}}
\vspace{-0.3cm}
\caption{Illustration showing that $f(x,y,z)=0$. Two rectangles of $[1.5,2.5]\times[-0.5,0.5]\times[-5,1]$ and $[0.5, 1.5]\times [0.5, 1.5]\times [2, 7]$ denote closed neighborhoods of $(2,0,-1)$ and $(1,1,4)$, respectively.}
\label{fig1}
\end{figure}

\end{example}

\section{Analytic implicit functions} \label{aif}

In this section, we derive a power series expansion of an implicit function which has one dependent variable, and any number of independent variables as before. 
The proposed method does not require the size estimate of Taylor coefficients
and even differentiability of $f$ as we have seen in Section 3.
We start with an assumption that the implicit function $y=g(x)$ for $f(x,y)=0$,
is analytic on a rectangle, which means that $g$ is analytic on some open neighborhood containing the rectangle.
By translation and dilation on the independent variables of $f$,
we may assume that the implicit function is analytic 
on $[-1-\delta,1+\delta]^n$ for some $\delta>0$
and write $g(x)=\sum_{\beta} c_\beta x^\beta$ 
which converges in $[-1-\delta,1+\delta]^n$, where the sum is the power series expansion of $g$ at the origin.   
Let $R=[-1,1]^n$ and put
\begin{equation*}
-1<-1+\Delta_x<-1+2\Delta_x<\cdots-1+(N-1)\Delta_x<-1+N\Delta_x=1 
\end{equation*}
be a partition on $[-1,1]$ 
$(\Delta_x=2/N)$. According to (\ref{data}) and (\ref{volume integration in poly}), prepare $d_{\alpha}$ from $f$ over  $R_\alpha\times I$, where
\begin{equation*}
	R_{\alpha}=\prod_{k=1}^n [-1+(\alpha_{k}-1)\Delta_x, -1+\alpha_k\Delta_x]
\end{equation*}
for $\alpha_k=1,2,\ldots,N$. 

We split $g$ into two parts: one is a partial sum in which the exponent of a monomial is dominated by an multi-index $N$, another is the remainder of the series of $g$, e.g.,
\begin{equation} \label{analytic series}
\begin{split}
	g(x)
	&=\sum_{0\le\beta< N}c_{\beta} x^{\beta} 
		+ \sum_{\text{remainder }\beta}c_{\beta} x^{\beta} \\
	&=g_N(x)+r_N(x),\quad\mbox{say}.
\end{split}	
\end{equation}
From analyticity, the series of $g$ converges uniformly and absolutely. There is a constant $C$ such that $|c_\beta x^\beta|\le C$ uniformly in $x\in [-1-\delta,1+\delta]^n$ for every $\beta$. 
So, $|c_\beta|(1+\delta)^{|\beta|}\le C$ for every $\beta$ (for the properties of real analytic functions of several variables, refer to \cite{kp2}).

In $r_N$ of (\ref{analytic series}), there is an index $\beta$ such that its component is greater than or equal to $N$. Assume that $\beta_1\ge N$. Then, for $x\in R$, 
\begin{equation} \label{vanish of tail}
\begin{split}
	\sum_{\text{remainder }\beta} |c_{\beta}||x^{\beta}|
		&\le C\sum_{\text{remainder }\beta} \frac{|x^{\beta}|}{(1+\delta)^{|\beta|}} \\
		&\le \sum_{\beta_1=N}^\infty\sum_{\substack{\beta_k=0\\k\ne1}}^\infty \frac{1}{(1+\delta)^{|\beta|}} \\
		&= \frac{1}{\delta^n(1+\delta)^{N-n}}
\end{split}
\end{equation}
which vanishes as $N\to\infty$.
Let $\epsilon>0$.
By the triangle inequality  and by (\ref{vanish of tail}), we have $\|r_N\|_{L^\infty(R)}<\epsilon$ and
%
\begin{equation} \label{errors2}
	\int_{R_\alpha}|r_N|\,dx \le |R_{\alpha}|\epsilon \le 2^n\epsilon
\end{equation}
for every sufficiently large $N$.

For a fixed $N$, we put $\tilde g_N(x)=\sum_{0\le\alpha<N}\tilde c_\alpha x^\alpha$, where 
$\tilde c_\alpha$ is calculated by solving
\begin{equation} \label{poly integration}
	\int_{R_\alpha}\tilde g_N\,dx = d_{\alpha}
\end{equation}
as the derivation of (\ref{1st coefficients in poly}).
By (\ref{2nd data}) and by (\ref{poly integration}),
\begin{equation} \label{pre_error}
\begin{split}
	\left|\int_{R_\alpha}g_N-\tilde g_N\,dx\right| 
	&=\left|\int_{R_\alpha}g_N\,dx - d_\alpha\right| \\
	&=\left|\int_{R_\alpha}g_N-g\,dx \right| \\
	&= \left|\int_{R_\alpha} r_N\,dx\right|,
\end{split}	
\end{equation}
where the last equality follows from (\ref{analytic series}). 
By the Riesz representation theorem for the Lebesgue spaces (the duality argument), we realize  the next lemma.

\begin{lemma} \label{uniform limit}
The function $\tilde g_N$ 
converges to $g$ weak-star in $L^\infty(R)$ as $N\to\infty$.
\end{lemma}

\begin{proof}
We recall (\ref{pre_error}).
By the triangle inequality,
\begin{equation} \label{weak-star}
\int_{R_\alpha}|\tilde g_N-g|dx\le\int_{R_\alpha}|\tilde g_N-g_N|dx + \int_{R_\alpha}|r_N|dx
\to0
\end{equation}
uniformly on $\alpha$ as $N\to\infty$. 
Here, the convergence of (\ref{weak-star}) follows from (\ref{errors2}).
Since the collection of finite linear combinations of characteristic functions $\chi_{R_\alpha}$ which are supported on grid blocks $R_\alpha$ is dense in $L^1(R)$,
the duality argument (\cite{folland,rudin}) from (\ref{pre_error}) and (\ref{weak-star}) for Lebesgue spaces yield that
$\tilde g_N$ goes to $g$ weak-star in $L^\infty(R)$ as $N\to\infty$. 
Therefore, the proof is complete.
\end{proof}

Let $V_N$ and $W$ be a matrix and tensor as in Theorem \ref{polynomial solution} with $V_N=V_{N_k}$ and $N= N_k$. By Lemma \ref{uniform limit}, we conclude the main theorem of this section. 

\begin{theorem}  \label{analytic solution}
If $y=g(x)$ is analytic on $R$ to $I$ such that $f(x,y)=0$, then
$\tilde g_N(x)=\sum_{0\le\alpha<N}\tilde c_\alpha x^\alpha$ with the coefficient tensor of
\begin{equation} \label{coefficients in poly}
	W_N\circ\Big[V_{N}^{-1}\stackrel{2\to n}{\cdot}
	\cdots
	V_{N}^{-1}\stackrel{2\to 2}{\cdot}
	V_{N}^{-1}\stackrel{2\to 1}{\cdot}
	\begin{pmatrix}
	d_\alpha
	\end{pmatrix}_\alpha\Big],
\end{equation}
converges to $g$ weak-star in $L^\infty(R)$ as $N\to\infty$.
\end{theorem}

As we have seen in Corollary \ref{poly ift}, 
the next corollary follows.

\begin{corollary}\label{analytic cor}
Suppose that $f$ is continuously differentiable such that $\partial_yf(0,0)\ne0$ with $f(0,0)=0$.
If $y=g(x)$ is analytic on $R$ to $I$ such that $f(x,y)=0$, then
then $\tilde g_N$ converges to $g$ weak-star in $L^\infty(R)$ as $N\to\infty$.
\end{corollary}

If $f$ is analytic such that $\partial_yf(0,0)\ne0$ with $f(0,0)=0$, then by the analytic version of the implicit function theorem, 
Corollary \ref{analytic cor} holds. We now present a numerical approximation of an analytic implicit function.

\begin{example} \label{sphere_ex}
Let $f:\mathbb{R}^3\to\mathbb{R}$ be given by
\begin{equation*}
\begin{split}
	f(x,y,z)=x^2 + y^2 + z^2 - 1
\end{split}
\end{equation*}	
with $f(0,0,1)=0$. 
We want to find a function of $x$ and $y$ that satisfies $f=0$ in some rectangle of $(0,0,1)$.
Since $f$ is analytic on a neighborhood of $(0,0,1)$ and $\partial_zf(0,0,1)=2>0$, 
by the analytic version of the implicit function theorem there is a rectangle of $(0, 0, 1)$ and 
a unique analytic implicit function for $z$ such that $f = 0$ in the rectangle.

Choose a rectangle $R\times I$, for example, $R=[-1/2, 1/2]\times[-1/2, 1/2]\ni(0,0)$ and $I=[0, 3/2]\ni1$, which are depicted in Figure \ref{fig2}.
In $R\times I$, readily $n_z(f)=1$ and with $N=6$, it follows that 
\begin{equation*}
V_{N} = 
		\arraycolsep=3.5pt\def\arraystretch{1.5}
		\begin{pmatrix}
		\Scale[0.8]{\Delta} & \Scale[0.8]{(-1/2+\Delta)^2-(-1/2)^2} & \Scale[0.8]{\cdots} & \Scale[0.8]{(-1/2+\Delta)^6-(-1/2)^6} \\
		\Scale[0.8]{\Delta} & \Scale[0.8]{(-1/2+2\Delta)^2-(-1/2+\Delta)^2} & \Scale[0.8]{\cdots} & \Scale[0.8]{(-1/2+2\Delta)^6-(-1/2+\Delta)^6} \\
		\Scale[0.8]{\vdots} & \Scale[0.8]{\vdots}  & \Scale[0.8]{\ddots} & \Scale[0.8]{\vdots} \\
		\Scale[0.8]{\Delta} & \Scale[0.8]{(-1/2+6\Delta)^2-(-1/2+5\Delta)^2} & \Scale[0.8]{\cdots} & \Scale[0.8]{(-1/2+6\Delta)^6-(-1/2+5\Delta)^6} \\
		\end{pmatrix},
\end{equation*}
where $\Delta=1/6$. From (\ref{data}),
the $6\times6$ matrix of $d_{i,j}$ is given by 
\begin{equation*}
\left(
\arraycolsep=1.7pt\def\arraystretch{0.9}
\begin{array}{cccccc}
\Scale[0.68]{\rotatebox{0}{$0.022341010165457$}} & \Scale[0.68]{\rotatebox{0}{$0.024192273341742$}} & \Scale[0.68]{\rotatebox{0}{$0.025065830482660$}} &
\Scale[0.68]{\rotatebox{0}{$0.025065830482660$}} & \Scale[0.68]{\rotatebox{0}{$0.024192273341742$}} & \Scale[0.68]{\rotatebox{0}{$0.022341010165457$}} \\
\Scale[0.68]{\rotatebox{0}{$0.024192273341742$}} & \Scale[0.68]{\rotatebox{0}{$0.025909540264518$}} & \Scale[0.68]{\rotatebox{0}{$0.026726310029805$}} &
\Scale[0.68]{\rotatebox{0}{$0.026726310029805$}} & \Scale[0.68]{\rotatebox{0}{$0.025909540264518$}} & \Scale[0.68]{\rotatebox{0}{$0.024192273341742$}} \\
\Scale[0.68]{\rotatebox{0}{$0.025065830482660$}} & \Scale[0.68]{\rotatebox{0}{$0.026726310029805$}} &\Scale[0.68]{\rotatebox{0}{$0.027518589197521$}} &
\Scale[0.68]{\rotatebox{0}{$0.027518589197521$}} & \Scale[0.68]{\rotatebox{0}{$0.026726310029805$}} & \Scale[0.68]{\rotatebox{0}{$0.025065830482660$}} \\
\Scale[0.68]{\rotatebox{0}{$0.025065830482660$}} & \Scale[0.68]{\rotatebox{0}{$0.026726310029805$}} & \Scale[0.68]{\rotatebox{0}{$0.027518589197521$}} &
\Scale[0.68]{\rotatebox{0}{$0.027518589197521$}} & \Scale[0.68]{\rotatebox{0}{$0.026726310029805$}} & \Scale[0.68]{\rotatebox{0}{$0.025065830482660$}} \\
\Scale[0.68]{\rotatebox{0}{$0.024192273341742$}} & \Scale[0.68]{\rotatebox{0}{$0.025909540264518$}} & \Scale[0.68]{\rotatebox{0}{$0.026726310029805$}} &
\Scale[0.68]{\rotatebox{0}{$0.026726310029805$}} & \Scale[0.68]{\rotatebox{0}{$0.025909540264518$}} & \Scale[0.68]{\rotatebox{0}{$0.024192273341742$}} \\
\Scale[0.68]{\rotatebox{0}{$0.022341010165457$}} & \Scale[0.68]{\rotatebox{0}{$0.024192273341742$}} & \Scale[0.68]{\rotatebox{0}{$0.025065830482660$}} &
\Scale[0.68]{\rotatebox{0}{$0.025065830482660$}} & \Scale[0.68]{\rotatebox{0}{$0.024192273341742$}} & \Scale[0.68]{\rotatebox{0}{$0.022341010165457$}}
\end{array}\right). 
\end{equation*}
By Corollary \ref{analytic cor}, the coefficient matrix of 
$z=\tilde g_N(x,y)=\sum_{0\le\alpha<6}c_{\alpha}x^{\alpha_1}y^{\alpha_2}$ which approximates $z=z(x,y)$ such that $f(x,y,z(x,y))=0$, is calculated as
\begin{equation} \label{app N=6} 
W\circ \Big[V_{N}^{-1}\stackrel{2\to 2}{\cdot}V_{N}^{-1} \stackrel{2\to 1}{\cdot}
		\begin{pmatrix}d_{\alpha}\end{pmatrix}_{\alpha} \Big] \\		
		= 
\arraycolsep=1pt\def\arraystretch{1}
\begin{matrix}
& \\
\left(\hspace{0.4em} \vphantom{ \begin{matrix} 12 \\ 12 \\ 12 \\ 12 \\ 12 \\ 12\end{matrix} } \right .
\end{matrix}
\hspace{-0.8em}
\begin{matrix}
\Scale[0.8]{1} & \Scale[0.8]{y} & \Scale[0.8]{y^2}  & \Scale[0.8]{y^3} & \Scale[0.8]{y^4}  & \Scale[0.8]{y^5} \\ \arrayrulecolor{gray}\hline \\[-3mm]
\Scale[0.8]{0.99997}  &  \Scale[0.8]{0}  &  \Scale[0.8]{-0.49880} & \Scale[0.8]{0}  & \Scale[0.8]{-0.14537}  &  \Scale[0.8]{0} \\
\Scale[0.8]{0} &  \Scale[0.8]{0} &  \Scale[0.8]{0} &  \Scale[0.8]{0}  & \Scale[0.8]{0}  &  \Scale[0.8]{0} \\
\Scale[0.8]{-0.49880}  &  \Scale[0.8]{0} &  \boxed{\Scale[0.8]{-0.24363}} &  \Scale[0.8]{0} &  \Scale[0.8]{-0.23108} &  \Scale[0.8]{0} \\
\Scale[0.8]{0} &  \Scale[0.8]{0} &  \Scale[0.8]{0} &  \Scale[0.8]{0}  & \Scale[0.8]{0}  &  \Scale[0.8]{0} \\
\Scale[0.8]{-0.14537}  &  \Scale[0.8]{0} &  \Scale[0.8]{-0.23108} &  \Scale[0.8]{0} &  \Scale[0.8]{-0.50484}  &  \Scale[0.8]{0} \\
\Scale[0.8]{0} &  \Scale[0.8]{0} &  \Scale[0.8]{0} &  \Scale[0.8]{0}  & \Scale[0.8]{0}  &  \Scale[0.8]{0}
\end{matrix}
\hspace{-0.2em}
\begin{matrix}
& \\
\left . \vphantom{ \begin{matrix} 12 \\ 12 \\ 12 \\ 12 \\ 12 \\ 12 \end{matrix} } \right )
\begin{matrix}
\Scale[0.8]{1} \\ \Scale[0.8]{x}  \\ \Scale[0.8]{x^2} \\ \Scale[0.8]{x^3} \\ \Scale[0.8]{x^4} \\ \Scale[0.8]{x^5}
\end{matrix}
\end{matrix},
\end{equation}
where $W=\big(ij\big)_{1\le i,j\le6}$
and, for example, $-0.24363$ in (\ref{app N=6}) is the coefficient of $x^2y^2$ a term of $g$. 
Both $\tilde g_N$ and $z(x,y)=(1-x^2-y^2)^{1/2}$ are plotted in Figure \ref{fig3}.

On the other hand, the coefficient matrix of the partial sum 
$g_N$ of the Taylor series of $z=(1-x^2-y^2)^{1/2}$ is as shown as
\begin{equation*} \label{taylor N=6} 
\arraycolsep=2.4pt\def\arraystretch{0.8}
\begin{matrix}
& \\
\left(\hspace{0.4em} \vphantom{ \begin{matrix} 12 \\ 12 \\ 12 \\ 12 \\ 12 \\ 12\end{matrix} } \right .
\end{matrix}
\hspace{-1.2em}
\begin{matrix}
\Scale[0.8]{1} & \Scale[0.8]{y} & \Scale[0.8]{y^2}  & \Scale[0.8]{y^3} & \Scale[0.8]{y^4}  & \Scale[0.8]{y^5} \\ \arrayrulecolor{gray}\hline \\[-3mm]
\Scale[0.8]{1}  &  \Scale[0.8]{0}  &  \Scale[0.8]{-0.5} & \Scale[0.8]{0}  & \Scale[0.8]{-0.125}  & \Scale[0.8]{0} \\
\Scale[0.8]{0} &  \Scale[0.8]{0} &  \Scale[0.8]{0} &  \Scale[0.8]{0}  & \Scale[0.8]{0}  &  \Scale[0.8]{0} \\
\Scale[0.8]{-0.5} &  \Scale[0.8]{0} &  \Scale[0.8]{-0.25} &  \Scale[0.8]{0} &  \Scale[0.8]{-0.1875} &  \Scale[0.8]{0} \\
\Scale[0.8]{0} &  \Scale[0.8]{0} &  \Scale[0.8]{0} &  \Scale[0.8]{0}  & \Scale[0.8]{0}  &  \Scale[0.8]{0} \\
\Scale[0.8]{-0.125}  &  \Scale[0.8]{0} &  \Scale[0.8]{-1875} &  \Scale[0.8]{0} &  \Scale[0.8]{-0.234375}  & \Scale[0.8]{0} \\
\Scale[0.8]{0} &  \Scale[0.8]{0} &  \Scale[0.8]{0} &  \Scale[0.8]{0}  & \Scale[0.8]{0}  &  \Scale[0.8]{0} \\
\end{matrix}
\hspace{-0.2em}
\begin{matrix}
& \\
\left . \vphantom{ \begin{matrix} 12 \\ 12 \\ 12 \\ 12 \\ 12 \\ 12 \end{matrix} } \right )
\begin{matrix}
\Scale[0.8]{1} \\ \Scale[0.8]{x}  \\ \Scale[0.8]{x^2} \\ \Scale[0.8]{x^3} \\ \Scale[0.8]{x^4} \\ \Scale[0.8]{x^5}
\end{matrix}
\end{matrix}.
\end{equation*}
As it is indicated on the graph of Figure \ref{fig4}, 
in the accuracy comparison between $\tilde g_N$ and $g_N$, 
the former is more accurate.
%
%
%
\begin{figure}[!ht]
\centerline{\includegraphics[clip, trim=0cm 8cm 0cm 8cm,width=0.70\textwidth]{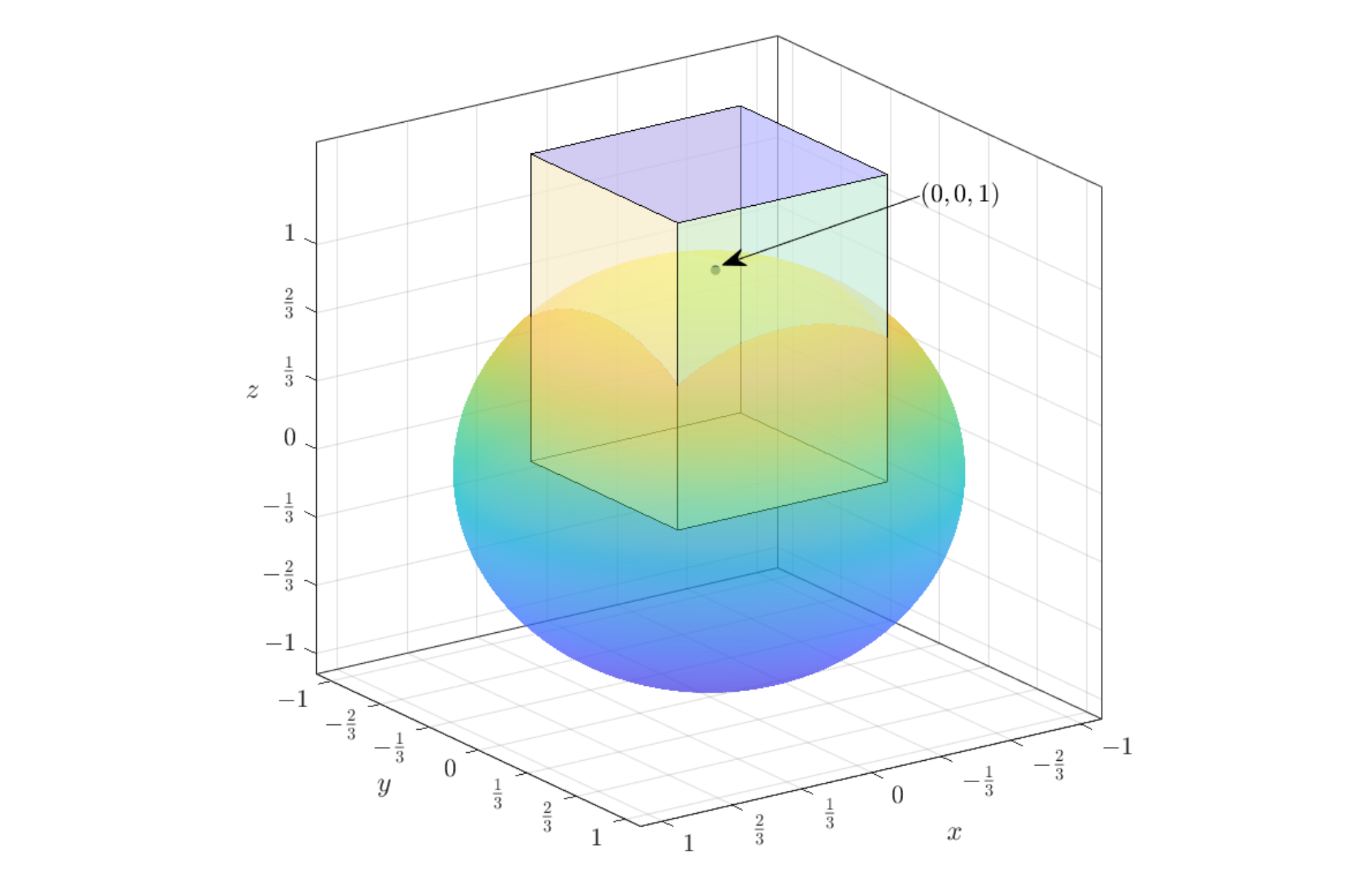}}
\vspace{-0.3cm}
\caption{Surface of $f(x,y,z)=x^2+y^2+z^2-1=0$. The rectangle $R\times I=[-0.5,0.5]\times[-0.5,0.5]\times[0,1.5]$ contains $(0,0,1)$.
}
\label{fig2}
\end{figure}
\begin{figure}[!ht] 
   \centering
     \makebox[\textwidth]{
    \subfloat[$z=\tilde g_N(x,y)$ with $N=6$]%
{\includegraphics[clip, trim=0cm 8cm 0cm 8cm,width=0.55\textwidth]{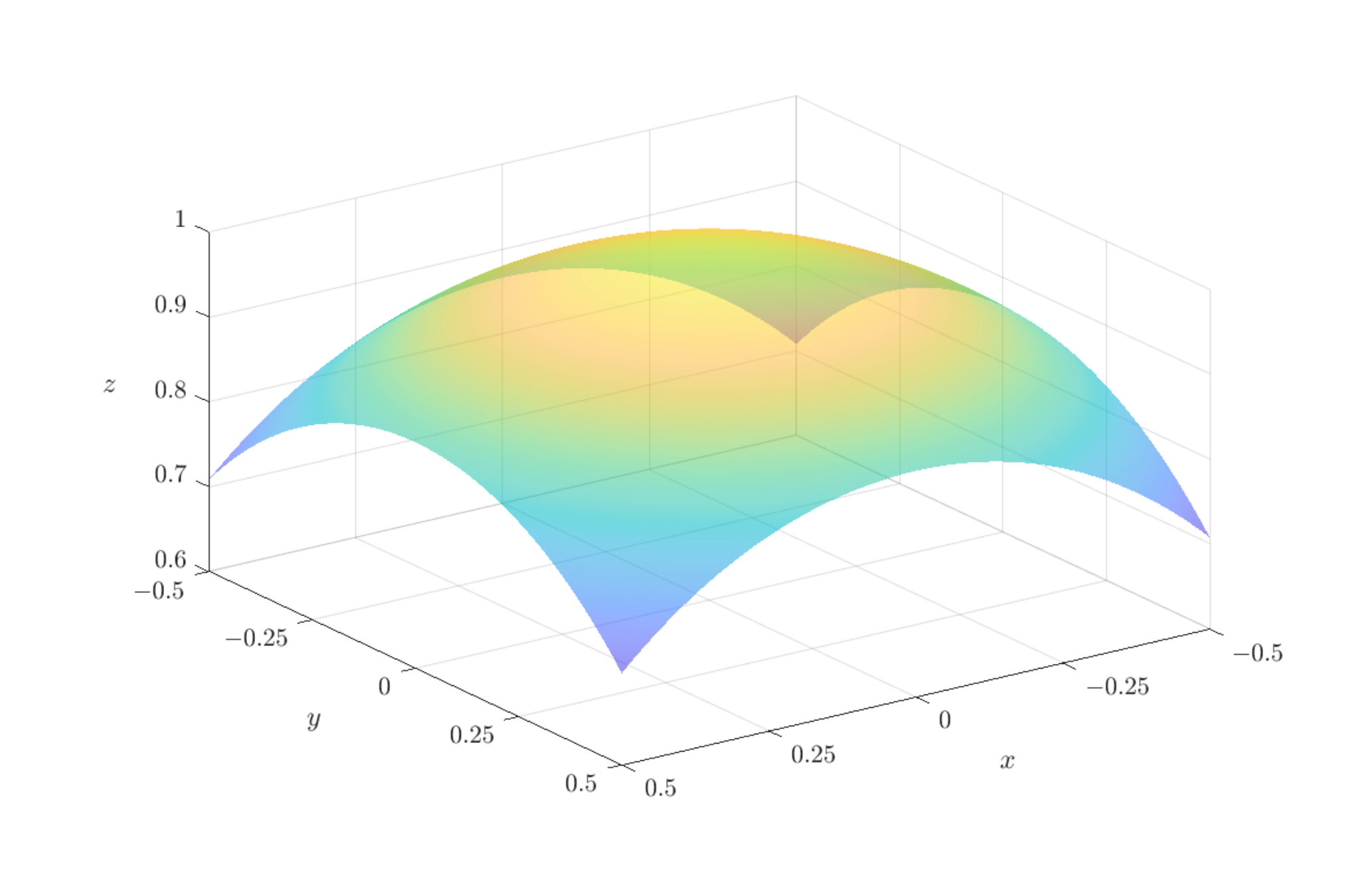} } \kern-0.4cm%
    \subfloat[$z(x,y)=(1-x^2-y^2)^{1/2}$]%
{\includegraphics[clip, trim=0cm 8cm 0cm 8cm,width=0.55\textwidth]{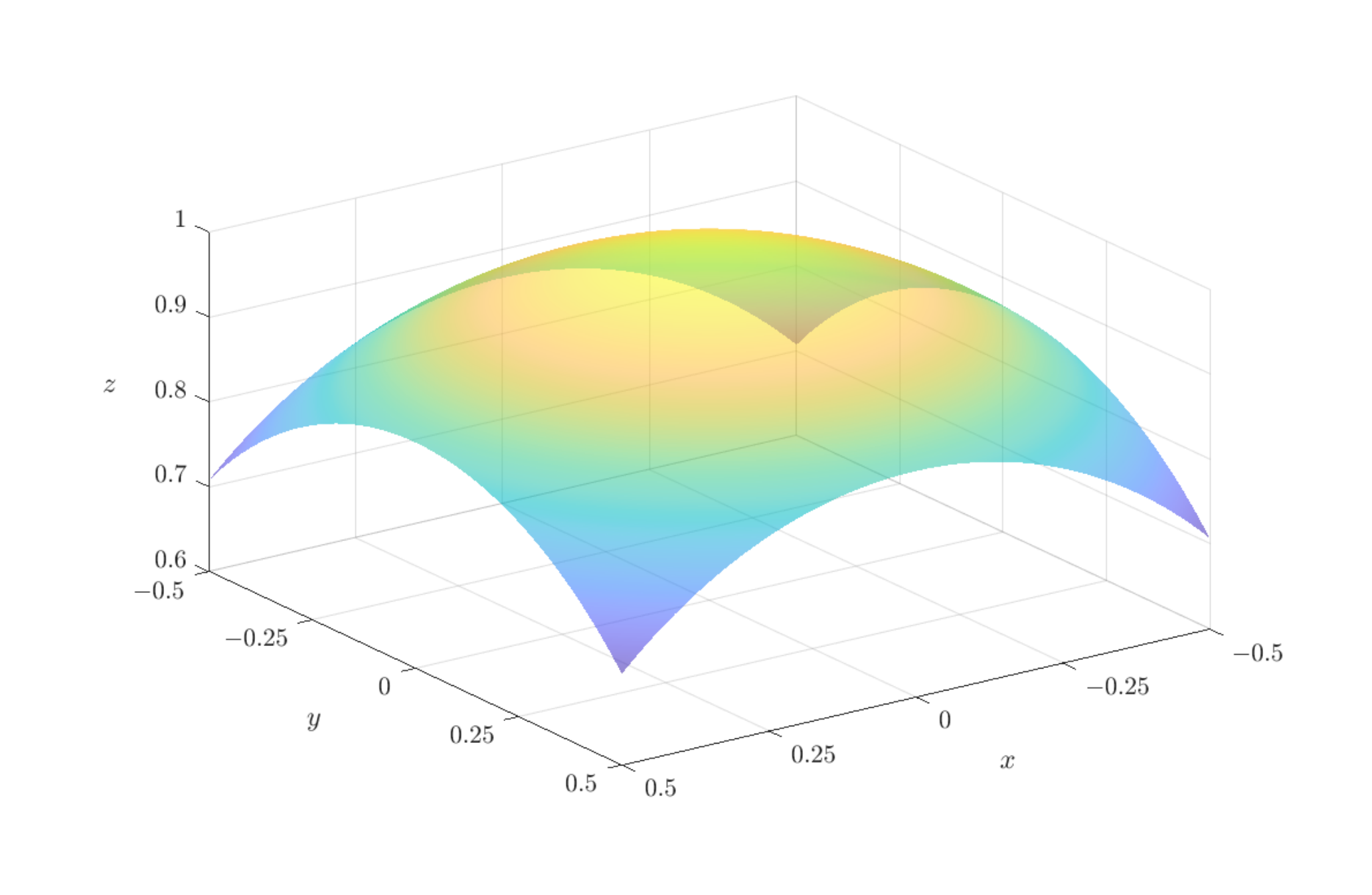} }%
    }
\vspace{-0,2cm}    
\caption{The comparison of $z=\tilde g_N(x,y)$ and $z=z(x,y)$ in $[-0.5,0.5]\times[-0.5,0.5]$.}%
    \label{fig3}%
\end{figure}
\begin{figure}[!ht]
\centerline{\includegraphics[clip, trim=0cm 8cm 0cm 8cm,width=0.70\textwidth]{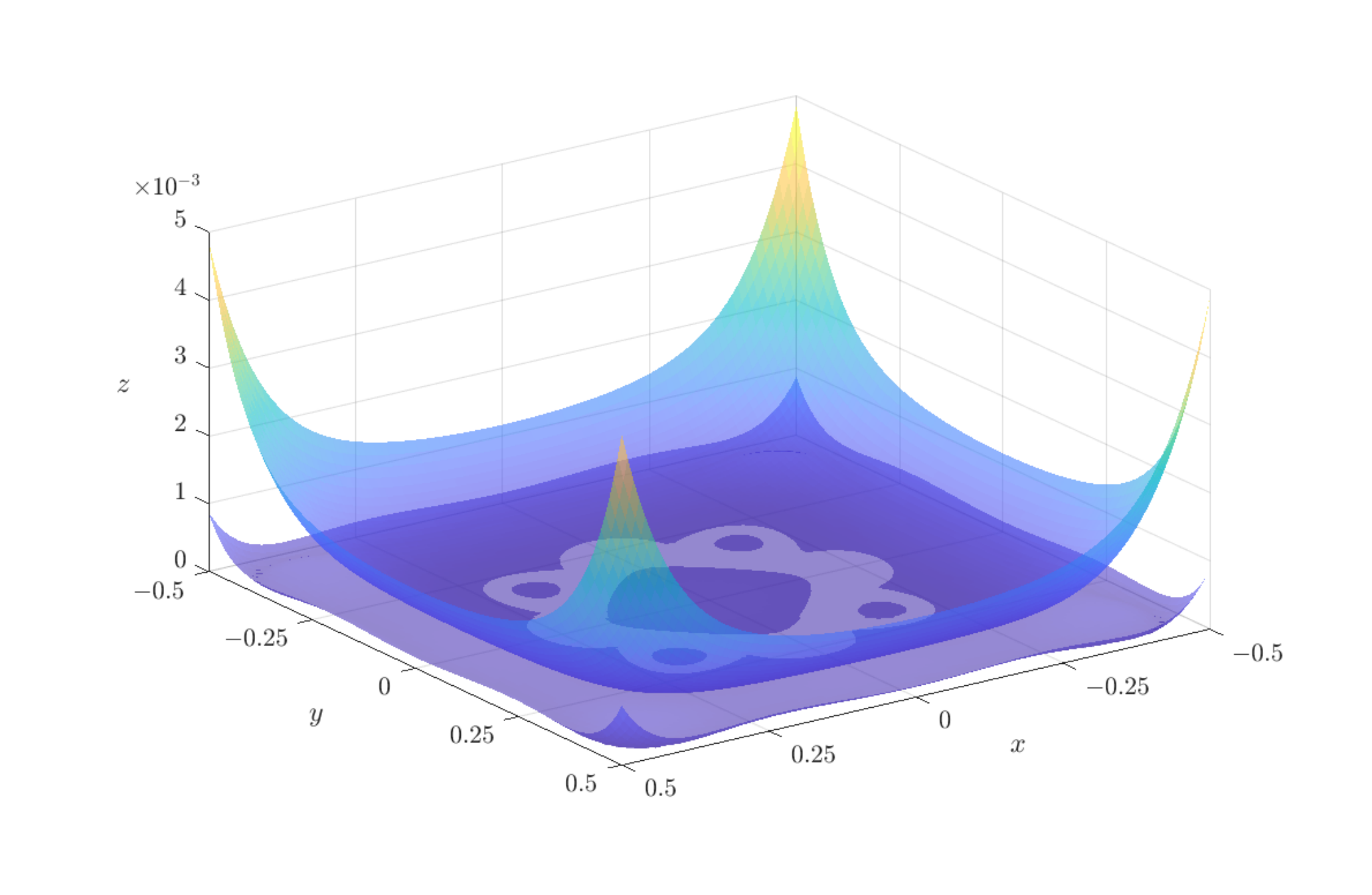}}
\vspace{-0.8cm}
\caption{ The comparison of $|z(x,y)-g_N(x,y)|$ and
$|z(x,y)-\tilde g_N(x,y)|$. 
The former is drawn on the top side and the latter on the bottom side.
}
\label{fig4}
\end{figure}
\end{example}

In general, the implicit function and inverse function theorems can be thought of as equivalent formulations of a similar basic idea. The next example leads to the polynomial approximation of an inverse function for Kepler’s equation.
When an initial value is given, the methods using iterative calculation of a trajectory that satisfies  Kepler's equation are widely used in practice. Although the methods using the formula by a function are very useful, which are formal infinite sums or depend on the expansion method of the Taylor series, there is no proper approach to achieve high-performance computing.
%
The following example provides 
more accurate numerical values than the other results on Kepler's equation.

\begin{example} \label{kepler's equation}
Let $M=E-\epsilon\sin(E)$ be the 
standard Kepler equation which fixes $(\pi,\pi)$ for every $0\le \epsilon\le1$, where
$M$, $E$, and $\epsilon$ are the mean anomaly, the eccentric anomaly, and the eccentricity, respectively.
Kepler's equation has a unique inverse function for every eccentricity, 
since $M$ is increasing monotonically.

Let $K(E,M)=E-\epsilon\sin(E)-M:\mathbb{R}^2\to\mathbb{R}$.
The Taylor series does not give a good approximation of the inverse function (the implicit function for $E$ of $K=0$) in the open neighborhood of $M=0$ and $\pi$, especially when $\epsilon=1$ because $\partial_EK$ vanishes.
(To approximate the inverse function, for example, \cite{kp} adopts the Lagrange inverse theorem
to obtain a formal series expansion and \cite{mat} 
provides the usability of a solution of Kepler's equation for $E$ as the most recent result.)

Choose a rectangle of $(\pi,\pi)$, for example, $[0,2\pi]\times[-\pi,3\pi]$ on which $n_E(K)=1$.
By Theorem \ref{analytic solution}, for $N=28$
we have an approximation of $E$ by $\tilde E(M)=\sum_{n=0}^{N-1}c_{n}(M-\pi)^n$ calculated as
\begin{equation*}
	\tilde E(M)=
\left(
\arraycolsep=1.5pt\def\arraystretch{2}
\begin{array}{cccccccccccccccccccccccccccc}
\Scale[0.8]{\rotatebox{90}{$3.14159265358979$}} &
\Scale[0.8]{\rotatebox{90}{$0.500000286048401{ }^{ }$}} &
\Scale[0.8]{\rotatebox{90}{$-6.39953880001189\times 10^{-13}$}} &
\Scale[0.8]{\rotatebox{90}{$0.0104167098436714$}} &
\Scale[0.8]{\rotatebox{90}{$3.52445662156362\times 10^{-12}$}} &
\Scale[0.8]{\rotatebox{90}{$0.000520826826664063$}} &
\Scale[0.8]{\rotatebox{90}{$-8.33322713405475\times 10^{-12}$}} &
\Scale[0.8]{\rotatebox{90}{$3.35293978405063\times 10^{-5}$}} &
\Scale[0.8]{\rotatebox{90}{$8.99171081830078\times 10^{-12}$}} &
\Scale[0.8]{\rotatebox{90}{$1.71069043293255\times 10^{-6}$}} &
\Scale[0.8]{\rotatebox{90}{$-6.35421688275359\times 10^{-12}$}} &
\Scale[0.8]{\rotatebox{90}{$1.35045779641222\times 10^{-6}$}} &
\Scale[0.8]{\rotatebox{90}{$9.47892775420489\times 10^{-13}$}} &
\Scale[0.8]{\rotatebox{90}{$-1.13558515793038\times 10^{-6}$}} &
\Scale[0.8]{\rotatebox{90}{$3.97758303345101\times 10^{-13}$}} &
\Scale[0.8]{\rotatebox{90}{$7.16518651568157\times 10^{-7}$}} &
\Scale[0.8]{\rotatebox{90}{$-3.64509053298631\times 10^{-13}$}} &
\Scale[0.8]{\rotatebox{90}{$-2.89790511992230\times 10^{-7}$}} &
\Scale[0.8]{\rotatebox{90}{$1.17074864690306\times 10^{-13}$}} &
\Scale[0.8]{\rotatebox{90}{$7.74501977110289\times 10^{-8}$}} &
\Scale[0.8]{\rotatebox{90}{$-1.92321275158235\times 10^{-14}$}} &
\Scale[0.8]{\rotatebox{90}{$-1.34780939066348\times 10^{-8}$}} &
\Scale[0.8]{\rotatebox{90}{$2.14518516845533\times 10^{-15}$}} &
\Scale[0.8]{\rotatebox{90}{$1.46671954613640\times 10^{-9}$}} &
\Scale[0.8]{\rotatebox{90}{$-1.16947164649083\times 10^{-16}$}} &
\Scale[0.8]{\rotatebox{90}{$-9.03766696253492\times 10^{-11}$}} &
\Scale[0.8]{\rotatebox{90}{$3.13882342334049\times 10^{-18}$}} &
\Scale[0.8]{\rotatebox{90}{$2.40613536296425\times 10^{-12}$}}
\end{array}\right)_{\kern-4pt \Scale[0.6]{1\!\times\!28}}\kern-4pt
	\arraycolsep=2pt\def\arraystretch{1.3}
	\begin{pmatrix}
	\Scale[0.8]{1}  \\ \Scale[0.8]{M-\pi} \\ \Scale[0.8]{(M-\pi)^2} \\ \Scale[0.8]{\vdots} \\ 
	\Scale[0.8]{(M-\pi)^{26}} \\ \Scale[0.8]{(M-\pi)^{27}}
	\end{pmatrix}_{\kern-4pt \Scale[0.6]{28\!\times\!1}}.  
\end{equation*}

In Figures \ref{figk1} and \ref{figk2},
$\tilde E$ and $K(M,\tilde E(M))$ are shown for $\epsilon=0.8$, $0.9$, and $1$
(for comparison with related work, e.g., refer to \cite{RP,mat}).
\begin{figure}[!ht]
\centerline{\includegraphics[clip, trim=0cm 7.5cm 0cm 7.5cm,width=0.70\textwidth]{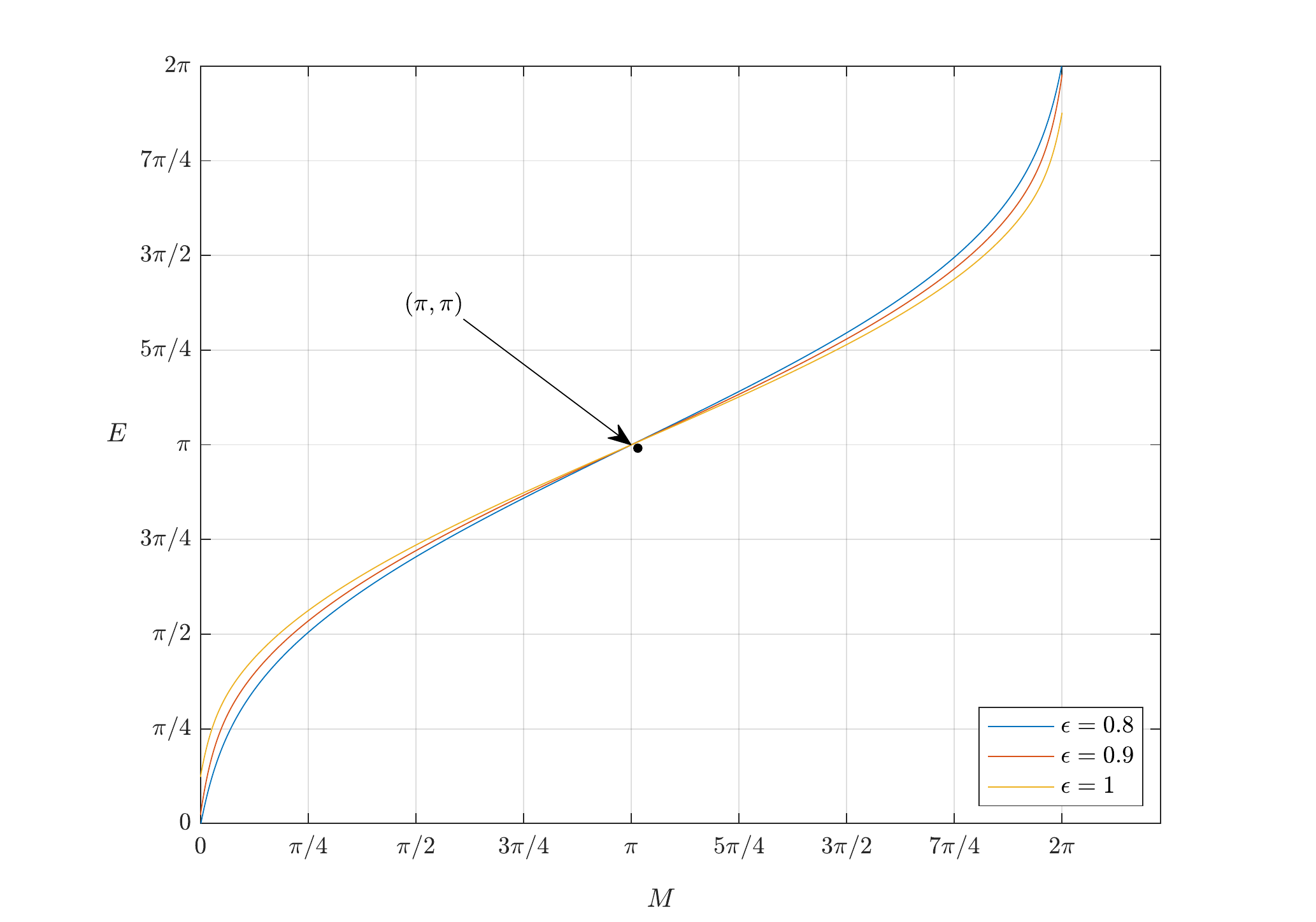}}
\caption{Polynomial approximation $\tilde E=\sum_{n=0}^{27}c_n(M-\pi)^n$ for the implicit function for $E$ of $K=0$ in $[0, 2\pi]$ with $\epsilon=0.8$, $0.9$, and $1$.}
\label{figk1}
\end{figure}
\begin{figure}[!ht]
\centerline{\includegraphics[clip, trim=0cm 7.5cm 0cm 7.5cm, width=0.70\textwidth]{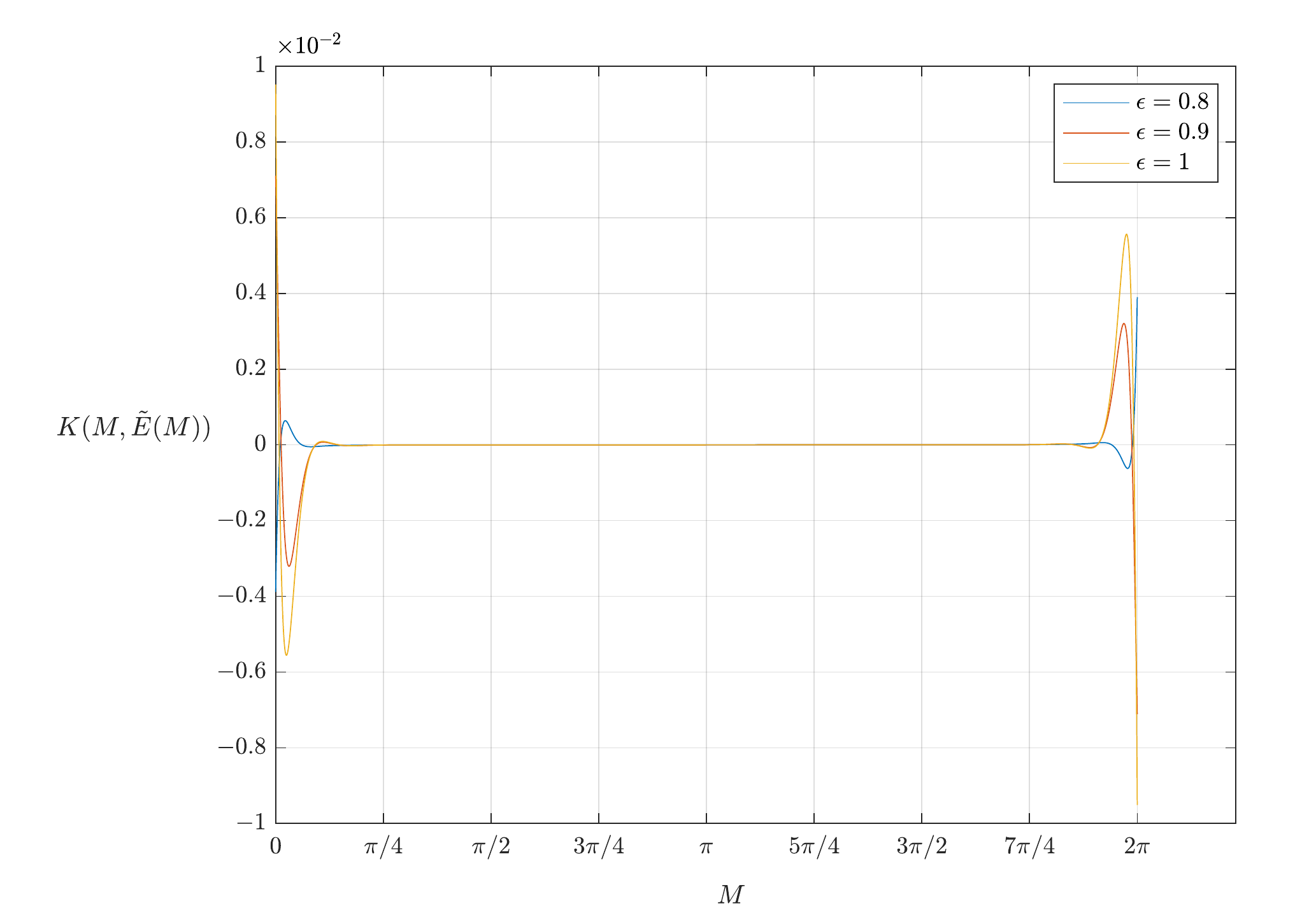}}
\caption{The non-vanishing values of $K(M,\tilde E(M))$ denote errors. 
}
\label{figk2}
\end{figure}
\end{example}

From Example \ref{kepler's equation}, we have seen that the inverse function for $y=f(x)$
is obtained by finding an implicit function for $x$ such that $F(x,y)=0$
after setting $F(x,y)=y-f(x)$. For a system of implicit equations, Theorem \ref{system of analytic solutions} will be applicable for constructing a vector-valued inverse function.
%

\section{System of analytic implicit functions}

In this section, we extend the results for a real-valued function in Section 4 to a vector-valued function.
We will prove it 
by a variable-reduction technique which is the process of eliminating the dependent variables one by one.
Let $\phi$ be a permutation on $\{1,2,\ldots,m\}$ that consists of all axis numbers for dependent variables and  put $I=\prod_{i=1}^m I_i\subset\mathbb{R}^m$, where $I_i$ is an interval. 
For a vector $b$ and rectangle $I$, the notations $b^{\{j\}}$ and $I^{\{j\}}$ denote the $j$th component and edge deletions of $b$ and $I$, respectively. 

\begin{assumption2}
For a continuous vector-valued function $f:\mathbb{R}^{n+m}\to\mathbb{R}^m$, 
the following conditions are satisfied in descending induction over the number of dependent variables:
there is a permutation $\phi$ and an $R\times I\subset\mathbb{R}^{n+m}$ such that  
\begin{itemize}
\item[$(i)$]  for $(x,y^{\{m\}}) \in R\times I^{\{m\}}$,
\begin{equation*} \label{criterion2-1}
	\sign_{y_{m}} f_{\phi(m)}(x,y^{\{m,\}},y_{m})
\end{equation*}
has only one jump discontinuity as a function of $y_{m}$ on $I_{m}$, 
\item[$(ii)$]  for $(x,y^{\{m,m-1\}}) \in R\times I^{\{m,m-1\}}$,
\begin{equation*} \label{criterion2-2}
	\sign_{y_{m-1}} f_{\phi(m-1)}(x,y^{\{m,m-1\}},y_{m-1},h_{\phi(m)}(x,y^{\{m\}}))
\end{equation*}
has only one jump discontinuity as a function of $y_{m-1}$ on $I_{m-1}$, where
$y_m=h_{\phi(m)}$ solves $f_{\phi(m)}(x,y^{\{m\}},y_m)=0$ in $R\times I^{\{m,m-1\}}$.
\end{itemize}
\end{assumption2}

Using $(i)$ and $(ii)$, by the intermediate value theorem, by the continuity of $f$, and by the uniqueness of jump discontinuity, we know that $y_m=h_{\phi(m)}$ and $y_{m-1}=h_{\phi(m-1)}$ are continuous, uniquely determined.
Hence, Assumption 2 is well explained by descending induction.

\begin{example} 
Let $f=(f_1,f_2):\mathbb{R}^3\to\mathbb{R}^2$ be given by $f_1(x,y,z)=y-1$, $f_2(x,y,z)=x$  with $f(0,1,1)=\mathbf{0}$.
By observation of $f_1$ and $f_2$, we easily know that there is no rectangle of $(0,1,1)$ on which $\sign_{x} f_1(x,y,z)$, $\sign_{z} f_1(x,y,z)$, $\sign_{y} f_2(x,y,z)$, and $\sign_{z} f_2(x,y,z)$ have only one jump discontinuity. 
This means that there is no $\phi$ on a set containing the $z$-axis number,  which satisfies Assumption 2.
On the other hand, since $\partial_yf_1(0,1,1)=1\ne0$, by the analytic version of the implicit function theorem, 
there is a rectangle of $(0,1,1)$ and a unique analytic function $y=y(x,z)$ such that $f_1=0$ in the rectangle
on which $\sign_{y} f_1(x,y,z)$ has only one jump discontinuity.
Moreover, 
\begin{equation*}
	\partial_x f_2(x,y(x,z),z)\big|_{(x,z)=(0,1)}
		= \partial_x f_2(0,1,1)+\partial_y f_2(0,1,1)\partial_x y(0,1) 
		=1\ne0.
\end{equation*}
Thus, $\sign_x f_2(x,y(x,z),z)$  has only one jump discontinuity in the rectangle.
Now the permutation $\phi$ is defined by $\phi(1)=2$ and $\phi(2)=1$, where $1$ and $2$ denote the axis numbers of $x$ and $y$, respectively, 
that satisfies Assumption 2 in some rectangle of $(0,1,1)$.
\end{example}

\begin{example} 
Let $f=(f_1,f_2):\mathbb{R}^3\to\mathbb{R}^2$ be defined by  
$f_1(x,y,z)=y-1$ and $f_2(x,y,z)=(z-1)^2$ with $f(0,1,1)=\mathbf{0}$.
%
%
Since $f_2$ is a squared function, $\sign_{z} f_2(x,y,z)$ has no jump discontinuity on any rectangle of $(0,1,1)$. 
However, if we consider $f_1$ and $\partial_z f_2$ instead of  $f_2$, then 
$\sign_{y}f_1$ and $\sign_{z}\partial_{z} f_2$ have only one jump discontinuity, for example, on $[-1,1]\times[-1,2]\times[-1,2]\ni(0,1,1)$.
We can find $\phi$ by defining $\phi(1)=1$ and $\phi(2)=2$, which satisfies Assumption 2, where $1$ and $2$ denote the axis numbers of $y$ and $z$, respectively. 
Note that $(f_1,f_2)$ and $(f_1,\partial_{z}f_2)$ have the same implicit function for $y$ and $z$ in the rectangle.
\end{example} 

We need a lemma to prove the main theorem of this section, which will be shown by using a method of variable elimination.

\begin{lemma} \label{system of analytic solutions-1}
If $y_i=h_{\phi(i)}$ satisfies
\begin{equation*}
f_{\phi(i)}(x,y^{\{m,m-1,\ldots,i\}},h_{\phi(i)}(x,y^{\{m,m-1,\ldots,i\}}),
h_{\phi(i+1)}(x,y^{\{m,m-1,\ldots,i+1\}}),\ldots,h_{\phi(m)}(x,y^{\{m\}}))=0
\end{equation*}
in $R\times I^{\{m,m-1,\ldots,i\}}$ $(i=m,m-1,\ldots,1)$, then
$y=g(x)$ such that $f=\mathbf{0}$, is continuous and
determined uniquely by 
\begin{equation} \label{variable elimination}
\begin{split}
	g_1(x) &= h_{\phi(1)}(x), \\
	g_2(x) &= h_{\phi(2)}(x,g_1(x)), \\
	&\;\;\vdots \\
	g_m(x) &= h_{\phi(m)}(x,g_1(x),g_2(x),\ldots,g_{m-1}(x))
\end{split}
\end{equation}
in $R$
\end{lemma}

\begin{proof}
For convenience, we assume that $\phi=id$.
By regarding $(x,y^{\{m\}})$ as the independent variables of $f$,
from $(i)$ there is a unique continuous function $y_m=h_m(x,y^{\{m\}}):R\times I^{\{m\}} \to I_m$ such that 
$f_m(x,y^{\{m\}},y_m)=0$ in $R\times I^{\{m\}}$.

Substitute $y_m=h_m(x,y^{\{m\}})$ into $f$ and put
$f(x,y^{\{m\}},h_m(x,y^{\{m\}}))=f^{[1]}(x,y^{\{m\}})$. Precisely, 
\begin{equation*}
\begin{split}
	f_1^{[1]}(x,y^{\{m\}}) &= f_1(x,y^{\{m\}},h_m(x,y^{\{m\}})), \\
		&\;\;\vdots \\
	 f_{m-1}^{[1]}(x,y^{\{m\}}) &= f_{m-1}(x,y^{\{m\}},h_m(x,y^{\{m\}})), \\
	f_{m}^{[1]}(x,y^{\{m\}}) &= 0
\end{split} 
\end{equation*}
on $R\times I^{\{m\}}$, where the $m-1$ nontrivial components of $f^{[1]}$ do not depend on the $y_m$-variable.

For every $(x,y^{\{m.m-1\}})\in R\times I^{\{m,m-1\}}$, 
the property of $(ii)$ shows that 
\begin{equation*} \label{pre-reduction process}
	\sign_{y_{m-1}}f_{m-1}^{[1]}(x,y^{\{m\}})
	= \sign_{y_{m-1}}f_{m-1}(x,y^{\{m\}},h_m(x,y^{\{m\}}))
\end{equation*}
has one jump discontinuity on $I_{m-1}$.
This also produces a unique continuous function 
$y_{m-1}=h_{m-1}(x,y^{\{m,m-1\}}):R\times I^{\{m,m-1\}}\to I_{m-1}$ such that 
$f^{[1]}_{m-1}(x,y^{\{m,m-1\}},y_{m-1})=0$.

Substitute $y_{m-1}=h_{m-1}(x,y^{\{m,m-1\}})$ into $f^{[1]}$ and similar to the argument above, put $f^{[1]}(x,y^{\{m,m-1\}},h_{m-1}(x,y^{\{m,m-1\}}))=f^{[2]}(x,y^{\{m,m-1\}})$, i.e.,
\begin{equation*}
\begin{split}
	 f_1^{[2]}(x,y^{\{m,m-1\}}) &= f_1^{[1]}(x,y^{\{m,m-1\}},h_{m-1}(x,y^{\{m,m-1\}})), \\%
	&\;\;\vdots \\
	f_{m-2}^{[2]}(x,y^{\{m,m-1\}}) &= f_{m-2}^{[1]}(x,y^{\{m,m-1\}}, h_{m-1}(x,y^{\{m,m-1\}})), \\%
	f_{m-1}^{[2]}(x,y^{\{m,m-1\}}) &= 0, \\
	f_{m}^{[2]}(x,y^{\{m,m-1\}}) &= 0
\end{split}	
\end{equation*}
on $R\times I^{\{m,m-1\}}$. Then the $m-2$ nontrivial components of $f^{[2]}$ clearly do not depend on the variables of $y_m$ and $y_{m-1}$.

Similarly, for $(x,y^{\{m,m-1,m-2\}})\in R\times I^{\{m,m-1,m-2\}}$, 
the inductively assumed property of $(ii)$ shows that 
\begin{equation*} \label{reduction process2}
\begin{split}
	\sign_{m-2}f_{m-2}^{[2]}(x,y^{\{m.m-1\}})
	&= \sign_{m-2}f_{m-2}^{[1]}(x,y^{\{m,m-1\}},h_{m-1}(x,y^{\{m,m-1\}})) \\
	&= \sign_{m-2}f_{m-2}(x,y^{\{m,m-1\}},h_{m-1}(x,y^{\{m,m-1\}}),h_m(x,y^{\{m\}}))
\end{split}	
\end{equation*}
has one jump discontinuity on $I_{m-2}$.
Again, we gain a unique continuous function 
$y_{m-2}=h_{m-2}(x,y^{\{m,m-1,m-2\}}):R\times I^{\{m,m-1,m-2\}}\to I_{m-2}$ which satisfies 
the implicit equation of $f^{[2]}_{m-2}(x,y^{\{m,m-1,m-2\}},y_{m-2})=0$.

Repeating the above variable-reduction process inductively until reaching $y_1$, 
we get a unique continuous function $y_1=h_1(x):R_n\to I_1$ such that $f^{[m-1]}_1(x,h_1(x))=0$. Eventually, the unique continuous function $y=g(x)$ for $f(x,y)=\mathbf{0}$ is calculated as 
\begin{equation} \label{variable elimination2}
\begin{split}
	y_1 &= h_1(x), \\
	y_2 &= h_2(x,y_1), \\
	&\;\;\vdots \\
	y_m &= h_m(x,y_1,\ldots,y_{m-1})
\end{split}
\end{equation}
in $R$. Therefore, we obtain (\ref{variable elimination}) and
the proof is complete.
\end{proof}

Although $y_i=h_{\phi(i)}$ in Lemma \ref{system of analytic solutions-1} has an integral representation as shown in (\ref{int repre}), which does not reveal itself algebraically. 
From (\ref{variable elimination}) and (\ref{variable elimination2}), if every $h_{\phi(i)}$ is analytic, then every $g_i$ is also analytic, and vice versa.
If these are analytic, then every $g_i$ is calculated as the limit of a sequence of multivariate polynomials in the weak-star topology on $L^\infty$.
The necessity of the implicit function theorem and the analyticity of implicit functions yield the following main theorem and corollary.

\begin{theorem} \label{system of analytic solutions}
If $f$ is continuously differentiable such that $|J_{f,y}(a,b)|\ne0$ with $f(a,b)=\mathbf{0}$ and
every component of an implicit function for $f=\mathbf{0}$ is analytic on a neighborhood of $(a,b)$, then the implicit function is calculated uniquely in some rectangle of $(a,b)$ as (\ref{variable elimination}).
\end{theorem}

Since the analytic version of the implicit function theorem guarantees a unique existence of a system of analytic implicit functions in some rectangle, we have a corollary.

\begin{corollary} \label{ift2assumption3}
Suppose that every component of $f$ is analytic such that $|J_{f,y}(a,b)|\ne0$ with $f(a,b)=\mathbf{0}$. 
Then the implicit function for $f=\mathbf{0}$ is calculated uniquely in some rectangle of $(a,b)$
as (\ref{variable elimination}) .
\end{corollary}

\begin{proof}[Proof of Theorem \ref{system of analytic solutions}]
It suffices to show that there is an $R\times I$ and
a permutation $\phi$ with which $y_i=h_{\phi(i)}$ is analytic on $R\times I^{\{m,m-1,\ldots,i\}}$ $(i=m,m-1,\ldots,1)$, which satisfies Assumption 2. Then by Lemma \ref{system of analytic solutions-1}, the desired conclusion follows.

Since that $|J_{f,y}(a,b)|\ne0$, the column vectors of $J_{f,y}(a,b)$ are not all zero.   
There is $i$ such that $\partial_{y_m}f_i(a,b)\ne0$
and define $\phi(m)=i$.
Thus there is a unique continuous function $y_m=h_i(x,y^{\{m\}})$
such that $f_i(x,y^{\{m\}},h_i(x,y^{\{m\}}))=0$ for some rectangle of $(a,b)$.
In addition, 
by Corollary \ref{analytic cor} with analyticity, $y_m=h_i$ is analytic on 
$R^{(1)}\times \prod_{k=1}^{m-1}I_k^{(1)}$ to  
$I_m^{(1)}$ for some rectangle $R^{(1)}\times \prod_{k=1}^{m}I_k^{(1)}\ni(a,b)$.

Put $f^{[1]}(x,y^{\{m\}})=f^{\{i\}}(x,y^{\{m\}},h_i(x,y^{\{m\}}))$, where $f^{\{i\}}$ is the $i$th component removal from $f$.
First, we prove that $|J_{f^{[1]},y^{\{m\}}}(a,b^{\{m\}})|\ne0$.
Since $\partial_{y_m}f_i(a,b)\ne0$, 
at $(a,b^{\{m\}})$ the determinant of 
\begin{equation*} \label{sub_jacobian}
	|J_{f^{[1]},y^{\{m\}}}| 
	=
	\arraycolsep=3pt\def\arraystretch{1.3}
	\begin{vmatrix}
	\partial_{y_1}f_1^{[1]} & \cdots & \partial_{y_{m-1}} f_1^{[1]}  \\
	\vdots & \ddots & \vdots \\
	\partial_{y_1} f_{i-1}^{[1]} & \cdots & \partial_{y_{m-1}} f_{i-1}^{[1]} \\
	\partial_{y_1} f_{i+1}^{[1]} & \cdots & \partial_{y_{m-1}} f_{i+1}^{[1]} \\
	\vdots & \ddots & \vdots \\
	\partial_{y_1} f_{m}^{[1]} & \cdots & \partial_{y_{m-1}} f_{m}^{[1]} 
	\end{vmatrix}
\end{equation*}
is equal to
\begin{equation} \label{calculation of determinant}
	\frac{(-1)^{i+m}}{\partial_{y_m} f_i(a,b)} 
	\arraycolsep=3pt\def\arraystretch{1.3}
	\begin{vmatrix}
	\partial_{y_1} f_1^{[1]} & \cdots & \partial_{y_{m-1}} f_1^{[1]} & \partial_{y_m} f_1(a,b) \\
	\vdots & \ddots & \cdots & \vdots & \\
	\partial_{y_1} f_{i-1}^{[1]} & \cdots & \partial_{y_{m-1}} f_{i-1}^{[1]} & \partial_{y_m} f_{i-1}(a,b) \\
	0 & \cdots  & 0 &  \partial_{y_m} f_i(a,b) \\
	\partial_{y_1} f_{i+1}^{[1]} & \cdots & \partial_{y_{m-1}} f_{i+1}^{[1]} & \partial_{y_m} f_{i+1}(a,b)\\
	\vdots & \ddots & \vdots  & \vdots \\
	\partial_{y_1} f_{m}^{[1]} & \cdots  & \partial_{y_{m-1}} f_{m}^{[1]} &  \partial_{y_m} f_m(a,b)
	\end{vmatrix}.
\end{equation}
Let $1\le k\le m-1$.  Subtract $\partial_{y_k}h_i$ times the last column of (\ref{calculation of determinant}) from the other columns.
Then the $k$th column of (\ref{calculation of determinant}) is calculated as
\begin{equation} \label{jacob_column}
	\begin{pmatrix}
	\partial_{y_k} f_1^{[1]} - \partial_{y_m} f_1(a,b)\partial_{y_k} h_i \\
	\vdots \\
	\partial_{y_k} f_{i-1}^{[1]} - \partial_{y_m} f_{i-1}(a,b)\partial_{y_k} h_i \\
	 -\partial_{y_m} f_i(a,b)\partial_{y_k} h_i \\
	\partial_{y_k} f_{i+1}^{[1]} - \partial_{y_m} f_{i+1}(a,b)\partial_{y_k} h_i \\
	\vdots \\
	\partial_{y_k} f_{m}^{[1]} - \partial_{y_m} f_m(a,b)\partial_{y_k} h_i
	\end{pmatrix}
	=
	\begin{pmatrix}
	\partial_{y_k} f_1(a,b) \\
	\vdots \\
	\partial_{y_k}f_{i-1}(a,b) \\
	-\partial_{y_m}f_i(a,b)\partial_{y_k}h_i \\
	\partial_{y_k}f_{i+1}(a,b) \\
	\vdots \\
	\partial_{y_k}f_{m}(a,b)
	\end{pmatrix}
\end{equation}
by the chain rule, $\partial_{y_k}f_j^{[1]}=\partial_{y_k}f_j+\partial_{y_m}f_j\partial_{y_k}h_i$.
Since $f_i(x,y^{\{m\}},h_i(x,y^{\{m\}}))=0$ for $(x,y^{\{m\}})\in R^1\times \prod_{k=1}^{m-1}I_k^1$, we get
\begin{equation*}
0=\partial_{y_k}f_i(x,y^{\{m\}},h_i(x,y^{\{m\}}))\Big\vert_{(x,y^{\{m\}})=(a,b^{\{m\}})}
=\big[\partial_{y_k}f^{[1]}_i-\partial_{y_m}f_i\partial_{y_k}h_i\big]_{(x,y)=(a,b)}.
\end{equation*}
This identity and the change of variables of $\partial_{y_k}f^{[1]}_i=\partial_{y_k}f_i+\partial_{y_m}f_i\partial_{y_k}h_i$
give
\begin{equation*}
\begin{split}
-\partial_{y_m} f_i(a,b)\partial_{y_k} h_i(a,b^{\{m\}}) 
	&=
	\partial_{y_k} f_i(a,b) - \partial_{y_k} f_i(x,y^{\{m\}},h_i(x,y^{\{m\}})) %
	\Big\vert_{(x,y^{\{m\}})=(a,b^{\{m\}})} \\
	&=\partial_{y_k} f_i(a,b).
\end{split}	
\end{equation*}
So, (\ref{jacob_column}) equals
\begin{equation*}
	\begin{pmatrix}
	\partial_{y_k} f_1(a,b) \\
	\vdots \\
	\partial_{y_k}f_{i-1}(a,b) \\
	\partial_{y_k}f_i(a,b)  \\
	\partial_{y_k}f_{i+1}(a,b) \\
	\vdots \\
	\partial_{y_k}f_{m}(a,b)
	\end{pmatrix}
\end{equation*}
for $1\le k\le m$. Hence,
\begin{equation*} 
	|J_{f^{[1]},y^{\{m\}}}(a,b^{\{m\}})|=\frac{(-1)^{i+m}}{\partial_{y_m}f_i(a,b)} |J_{f,y}(a,b)| \ne 0.
\end{equation*}

Next, since $f^{[1]}$ satisfies $|J_{f^{[1]},y^{\{m\}}}(a,b^{\{m\}})|\ne0$ with $f^{[1]}(a,b^{\{m\}})=\mathbf{0}$, we can take $i'$ such that $\partial_{y_{m-1}}f^{[1]}_{i'}(a,b^{\{m\}})\ne0$. Here,  $i'\ne i$; 
therefore, put $\phi(m-1)=i'$.
As before, 
by Corollary \ref{analytic cor} there is a unique analytic function $y_{m-1}=h_{i'}(x,y^{\{m,m-1\}}):R^{(2)}\times \prod_{k=1}^{m-2}I_k^{(2)}\to I_{m-1}^{(2)}$ for some rectangle $R^{(2)}\times \prod_{k=1}^{m-1}I_k^{(2)}\ni(a,b^{\{m\}})$ in which $f^{[1]}_{i'}(x,y^{\{m,m-1\}},h_{i'}(x,y^{\{m,m-1\}}))=0$. 

Put $f^{[2]}(x,y^{\{m,m-1\}})=f^{[1]\{i'\}}(x,y^{\{m,m-1\}},h_{i'}(x,y^{\{m,m-1\}}))$, where $f^{[1]\{i'\}}$ indicates the $i'$th component removal from $f^{[1]}$.
By the exact same method above, 
\begin{equation*}
|J_{f^{[2]},y^{\{m,m-1\}}}(a,b^{\{m,m-1\}})|\ne0.
\end{equation*} 
In descending induction, $\phi$ is well defined as a permutation on $\{1,2,\ldots,m\}$ and take 
\begin{equation*}
R = \bigcap_{k=1}^mR^{(k)}\quad \mbox{and}\quad
I =  \bigcap_{k=1}^m I_1^{(k)}\times\bigcap_{k=1}^{m-1} I_2^{(k)}\times 
\cdots \times\bigcap_{k=1}^2I_{m-1}^{(k)}\times  I_m^{(1)}
\end{equation*}
which contain $a$ and $b$, respectively. Finally, we will prove that $f$ satisfies Assumption 2.
Since $y_i=h_{\phi(i)}(x,y^{\{m,m-1,\ldots,i\}})$ is a unique analytic function such that 
\begin{equation*}
\begin{split}
	f_{\phi(i)}(x,&y^{\{m,m-1,\ldots,i\}},
		h_{\phi(i)}(x,y^{\{m,m-1,\ldots,i\}}),\ldots,h_{\phi(m)}(x,y^{\{m\}})) \\
		&=f^{[m-i]}_{\phi(i)}(x,y^{\{m,m-1,\ldots,i\}},h_{\phi(i)}(x,y^{\{m,m-1,\ldots,i\}})) = 0 
\end{split}
\end{equation*}
in $R\times I^{\{m,m-1,\ldots,i\}}$, for every $(x,y^{\{m,m-1\ldots,i\}})\in R\times I^{\{m,m-1,\ldots,i\}}$ there is 
\begin{equation*}
	(y_{i+1},\ldots,y_m)
		=(h_{\phi(i+1)}(x,y^{\{m,m-1,\ldots,i+1\}}),\ldots,h_{\phi(m)}(x,y^{\{m\}}))\in I^{\{i,i-1,\ldots,1\}}
\end{equation*}
so that 
\begin{equation*}
	\sign_{y_i}f_{\phi(i)}(x,y^{\{m,m-1,\ldots,i\}},y_i,y_{i+1},\ldots,y_{m})
\end{equation*}
has only one jump discontinuity on $I_i^{(m-i+1)}$ for $i=m,m-1,\ldots,1$. 
Therefore, the proof is complete.
\end{proof}

The following two examples of implicit functions having three variables and two equations will serve to illustrate Theorem \ref{system of analytic solutions} and Corollary \ref{ift2assumption3}.  In the first example, an implicit function will be approximated 
when the Jacobian matrix is non-degenerate, while the second is an example which has a degenerate Jacobian matrix. 

\begin{example} \label{nondege}
Let $f=(f_1,f_2):\mathbb{R}^3\to\mathbb{R}^2$ be defined by $f_1(x,y,z)=x + y^2 + z^3 - 6$, $f_2(x,y,z)=x^3y - z - 1$ with $f(1,2,1)=\mathbf{0}$. We focus on calculating $g(x)=(g_1(x),g_2(x))$ with $y=g_1(x)$ and $z=g_2(x)$ such that $f=\mathbf{0}$ in an interval of $x=1$. 
Since $f_1$ and $f_2$ are analytic near $(x,y)=(1,2)$ and $|J_{f,(y,z)}(1,2,1)|=-7\ne0$,
by Corollary \ref{ift2assumption3} 
there is a rectangle, for example, $R=[0.5,1.5]\ni x=1$ and $I=[1.5,2.5]\times[-2,8]\ni(y,z)=(2,1)$ on which both $\sign_z f_1(x,y,z)$ and  $\sign_z f_2(x,y,z)$ have only one jump discontinuity. 
(The surfaces $f_1=0$ and $f_2=0$ are illustrated in Figure \ref{fig5-1}, in which $R\times I$ is also depicted.)

First, by observation, $n_z(f_2)=-1$. For $N=4$,  
$z=h_2(x,y)=\sum_{0\le\alpha<4}c_{\alpha}(x-1)^{\alpha_1}(y-2)^{\alpha_2}$ approximates $z=z(x,y)$ such that $f_2(x,y,z(x,y))=0$, where $c_{\alpha}$ is the $(\alpha+1)$th component of 
\begin{equation} \label{compare10} 
\arraycolsep=4pt\def\arraystretch{0.9}
\begin{matrix}
& \\
\left(\hspace{0.0em} \vphantom{ \begin{matrix} 12 \\ 12 \\ 12 \\ 12 \end{matrix} } \right .
\end{matrix}
\hspace{-1.2ex}
\begin{matrix}
\Scale[0.8]{1} & \Scale[0.8]{y-2} & \Scale[0.8]{(y-2)^2}  & \Scale[0.8]{(y-2)^3} & \\ \arrayrulecolor{gray}\hline \\[-4mm]
\Scale[0.8]{0.999999999999997}   & \Scale[0.8]{ 0.999999999999976}  & \Scale[0.8]{-0.000000000000017}  & \Scale[0.8]{0.000000000000259} \\
\Scale[0.8]{6.000000000000016}   & \Scale[0.8]{3.000000000000230}  & \Scale[0.8]{0.000000000000490}  & \Scale[0.8]{-0.000000000002899} \\
\Scale[0.8]{6.000000000000000}   & \Scale[0.8]{3.000000000000279}  & \Scale[0.8]{0.000000000000160}  & \Scale[0.8]{-0.000000000002686} \\
\Scale[0.8]{1.999999999999848}   & \Scale[0.8]{0.999999999998846}  & \Scale[0.8]{-0.000000000003098}  & \Scale[0.8]{0.000000000015994}
\end{matrix}
\hspace{-0.2em}
\begin{matrix}
& \\
\left . \vphantom{ \begin{matrix} 12 \\ 12 \\ 12 \\ 12  \end{matrix} } \right )
\begin{matrix}
\Scale[0.8]{1} \\ \Scale[0.8]{x-1}  \\ \Scale[0.8]{(x-1)^2} \\ \Scale[0.8]{(x-1)^3} 
\end{matrix}
\end{matrix}	%
\end{equation}
($z=h_2(x,y)$ is drawn in Figure \ref{fig6}).
Furthermore, the coefficients of $z=z(x,y)$ (which is a polynomial in two variables) are given by
\begin{equation*} 
\arraycolsep=3pt\def\arraystretch{1}
\begin{matrix}
& \\
\left(\kern-4pt \vphantom{ \begin{matrix} 12 \\ 12 \\ 12 \\ 12 \end{matrix} } \right .
\end{matrix}
\hspace{-0pt}
\begin{matrix} 
\Scale[0.8]{1} & \Scale[0.8]{y-2} & \Scale[0.8]{(y-2)^2}  & \Scale[0.8]{(y-2)^3} & \\ \arrayrulecolor{gray} \hline \\[-4mm]
\Scale[0.8]{1}   & \Scale[0.8]{1}  & \Scale[0.8]{0}  & \Scale[0.8]{0} \\
\Scale[0.8]{6}   & \Scale[0.8]{3}  & \Scale[0.8]{0}  & \Scale[0.8]{0} \\
\Scale[0.8]{6}   & \Scale[0.8]{3}  & \Scale[0.8]{0}  & \Scale[0.8]{0} \\
\Scale[0.8]{2}   & \Scale[0.8]{1}  & \Scale[0.8]{0}  & \Scale[0.8]{0}
\end{matrix}
\hspace{-0.4em}
\begin{matrix}
& \\
\left . \vphantom{ \begin{matrix} 12 \\ 12 \\ 12 \\ 12  \end{matrix} } \right )
\begin{matrix}
\Scale[0.8]{1} \\ \Scale[0.8]{x-1}  \\ \Scale[0.8]{(x-1)^2} \\ \Scale[0.8]{(x-1)^3} 
\end{matrix}
\end{matrix}\kern-6pt	%
\end{equation*}
which is comparable to (\ref{compare10}).

According to (\ref{variable elimination}), put
$f_1^{[1]}(x,y)=f_1(x,y,h_2(x,y))=0$ (as shown in Figure \ref{fig7}). Since $\sign_y f_1^{[1]}(x,y)$
has only one jump discontinuity and $n_y(f_1^{[1]})=1$, 
with $N=25$ we gain $y=h_1(x)=\sum_{k=0}^{24}c_k(x-1)^k$ in $[-0.5,2]\times[0,5]$, which is an approximation of $y=y(x)$ such that $f_1^{[1]}(x,y(x))=0$,
where 
\begin{equation*}
	h_1(x) = 
\arraycolsep=2.2pt\def\arraystretch{2}	
\left(
\begin{array}{ccccccccccccccccccccccccc}
  \Scale[0.8]{\rotatebox{90}{$2.0004078854$}} &
  \Scale[0.8]{\rotatebox{90}{$-2.7149208367$}} &
  \Scale[0.8]{\rotatebox{90}{$-4.9633925567$}} &
  \Scale[0.8]{\rotatebox{90}{$19.6874353356$}} &
  \Scale[0.8]{\rotatebox{90}{$25.2217132975$}} &
  \Scale[0.8]{\rotatebox{90}{$-235.2544997108$}} &
  \Scale[0.8]{\rotatebox{90}{$-40.2428746167$}} &
  \Scale[0.8]{\rotatebox{90}{$1987.0933621889$}} &
  \Scale[0.8]{\rotatebox{90}{$-94.7436991654$}} &
  \Scale[0.8]{\rotatebox{90}{$-11429.1583398359$}} &
  \Scale[0.8]{\rotatebox{90}{$-1087.1586128315$}} &
  \Scale[0.8]{\rotatebox{90}{$44047.6659489883$}} &
  \Scale[0.8]{\rotatebox{90}{$15624.2539582495$}} &
  \Scale[0.8]{\rotatebox{90}{$-110502.7914796270$}} &
  \Scale[0.8]{\rotatebox{90}{$-71424.8415543300$}} &
  \Scale[0.8]{\rotatebox{90}{$170165.9282511103$}} &
  \Scale[0.8]{\rotatebox{90}{$167091.0376411914$}} &
  \Scale[0.8]{\rotatebox{90}{$-138031.7528541524$}} &
  \Scale[0.8]{\rotatebox{90}{$-213411.3308704084$}} &
  \Scale[0.8]{\rotatebox{90}{$22951.9362353907$}} &
  \Scale[0.8]{\rotatebox{90}{$137736.7774057497$}} &
  \Scale[0.8]{\rotatebox{90}{$43043.3433440683$}} &
  \Scale[0.8]{\rotatebox{90}{$-30478.2871965461$}} &
  \Scale[0.8]{\rotatebox{90}{$-22061.1876771915$}} &
  \Scale[0.8]{\rotatebox{90}{$-3982.2743839354$}} 
\end{array}\right)_{\kern-4pt \Scale[0.6]{1\!\times\!25}}\kern-4pt
	\arraycolsep=3pt\def\arraystretch{1.1}
	\begin{pmatrix} 1 \\
	 x-1 \\
	 (x-1)^2 \\
	 \vdots \\
	 (x-1)^{24} 
	 \end{pmatrix}_{\kern-3pt \Scale[0.6]{25\!\times\!1}}
\end{equation*}
(shown in Figure \ref{fig8}).
Finally, 
\begin{equation*}
\begin{split}
y=g_1(x) &= h_1(x) \\
z=g_2(x) &= h_2(x, h_1(x))
\end{split}
\end{equation*}
are the approximations of $y=y(x)$ and $z=z(x)$ such that $f(x,y(x),z(x))=\mathbf{0}$, respectively, in $R$. 
In addition, 
the quantities 
$f_1(x,g_1(x), g_2(x))$ and $f_2(x, g_1(x), g_2(x))$ are illustrated in Figure \ref{fig8-1} and the approximated curve $x\mapsto(g_1(x),g_2(x))$ for $f=\mathbf{0}$ is shown in Figure \ref{fig9}.
\begin{figure}[!ht]
\centerline{\includegraphics[clip, trim=0cm 7.5cm 0cm 7.5cm, width=0.70\textwidth]{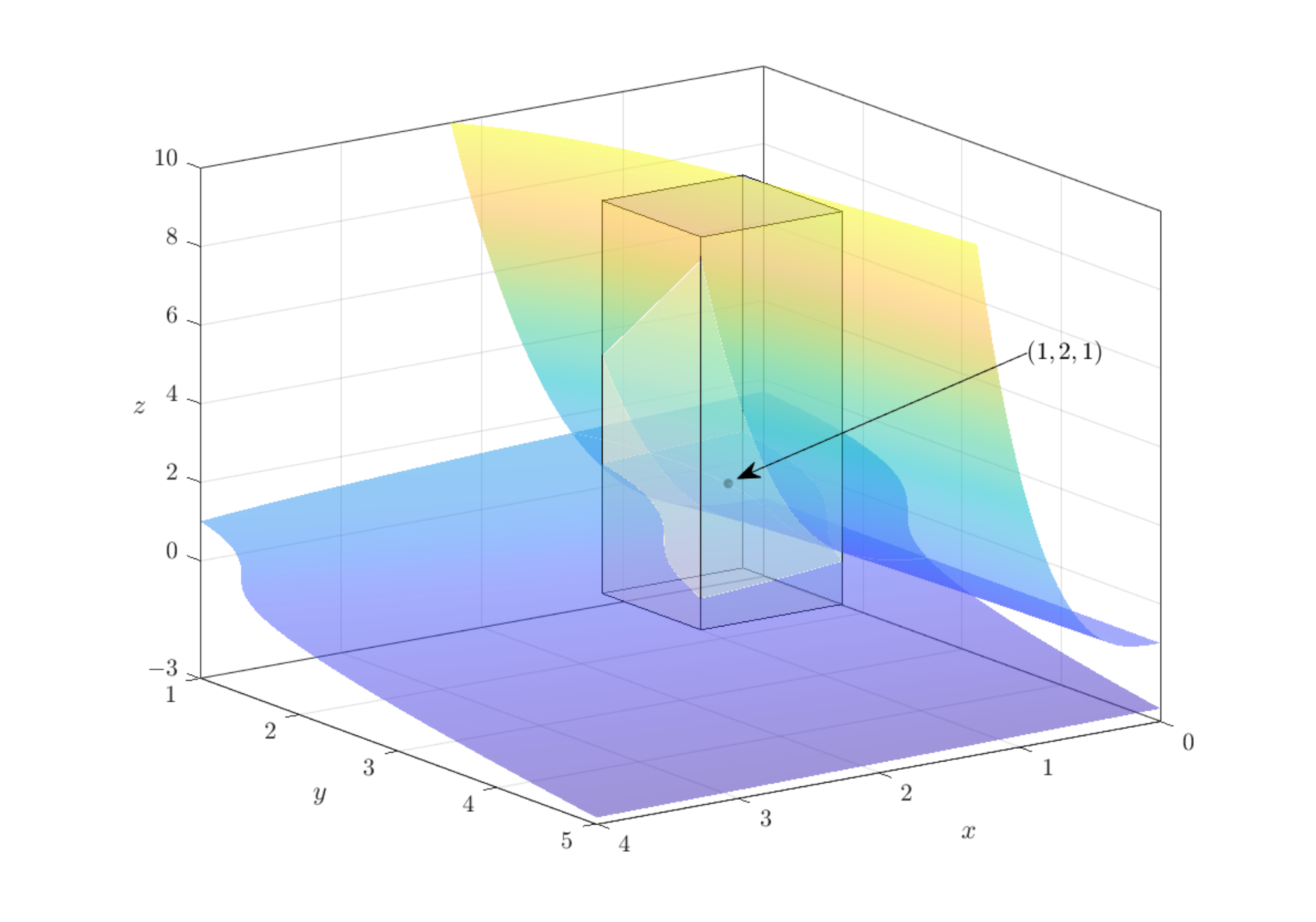}}
\vspace{-0.5cm}
\caption{The surfaces of $f_1(x,y,z)=x + y^2 + z^3 - 6=0$ and $f_2(x,y,z)=x^3y - z - 1=0$ and $R\times I$.}
\label{fig5-1}
\end{figure}
\begin{figure}[!ht]
\centerline{\includegraphics[clip, trim=0cm 7.5cm 0cm 7.5cm, width=0.70\textwidth]{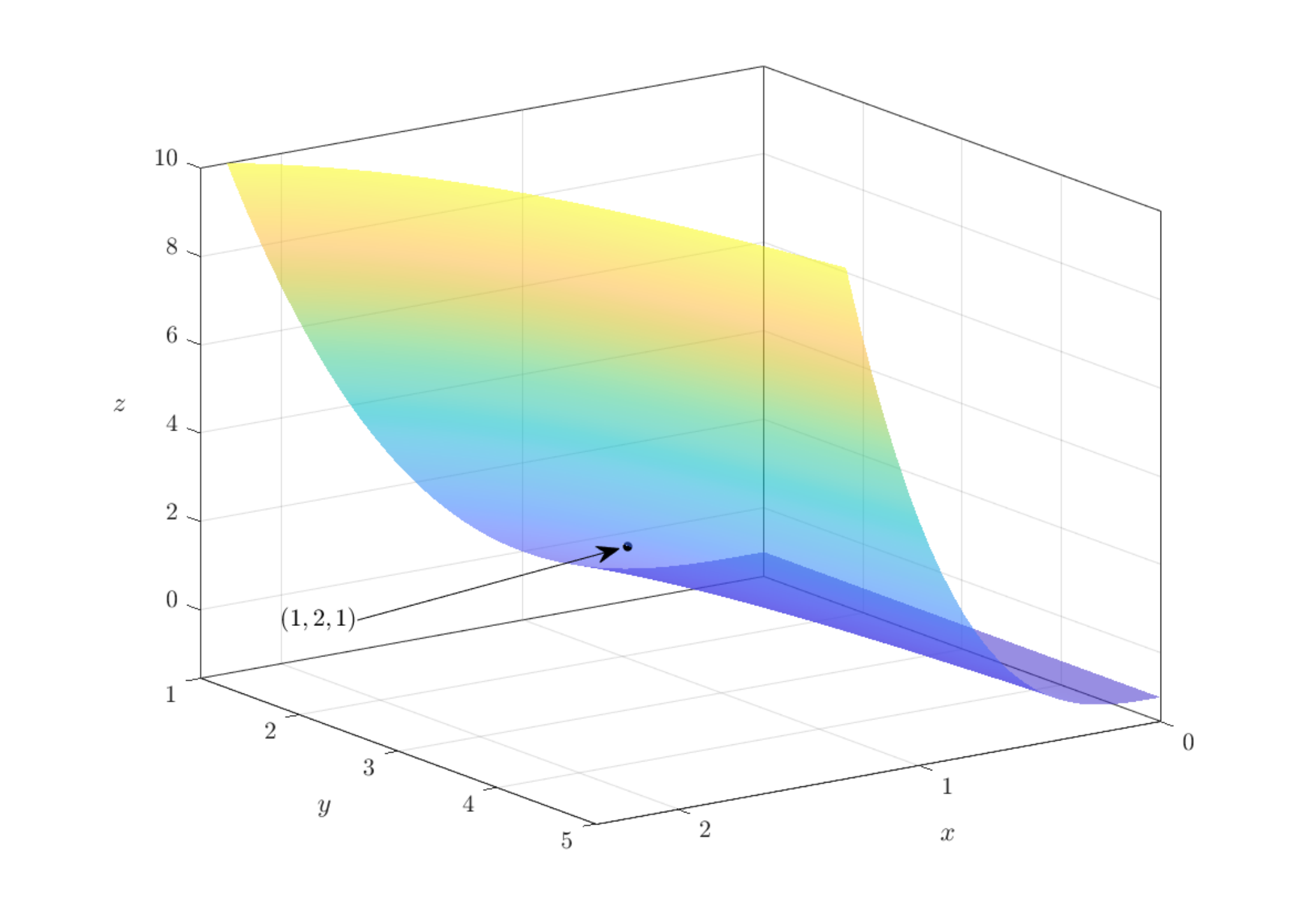}}
\vspace{-0.5cm}
\caption{The surface $z=h_2(x,y)=\sum_{0\le\alpha<4}c_{\alpha}(x-1)^{\alpha_1}(y-2)^{\alpha_2}$.}
\label{fig6}
\end{figure}
\begin{figure}[!ht]
\centerline{\includegraphics[clip, trim=0cm 7.5cm 0cm 7.5cm, width=0.70\textwidth]{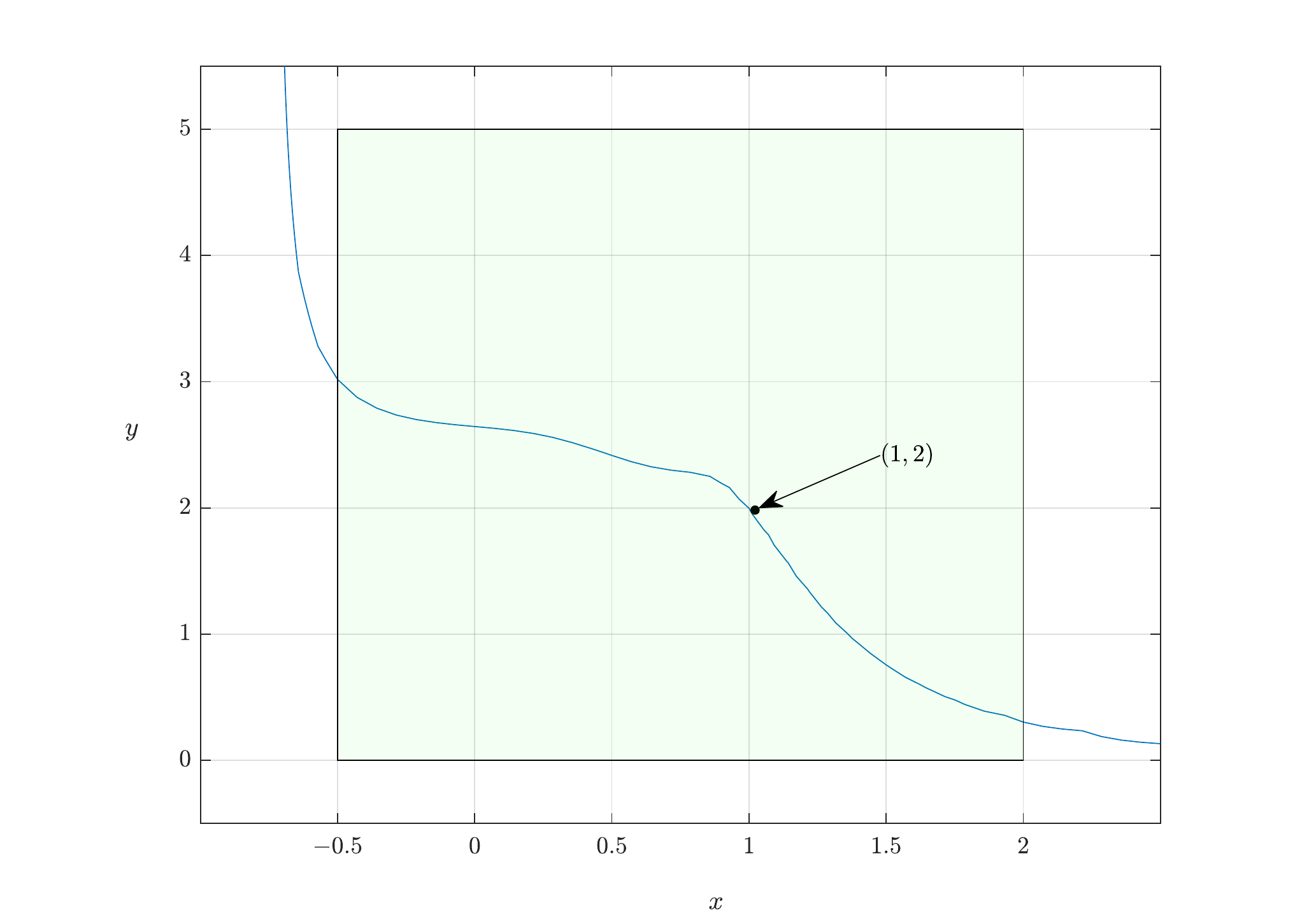}}
\caption{A graphical plot of the implicit function of $f_1(x,y,h_2(x,y))=0$ in $[-0.5,2]\times[0,5]$.}
\label{fig7}
\end{figure}
\begin{figure}[!ht]
\centerline{\includegraphics[clip, trim=0cm 7.5cm 0cm 7.5cm, width=0.70\textwidth]{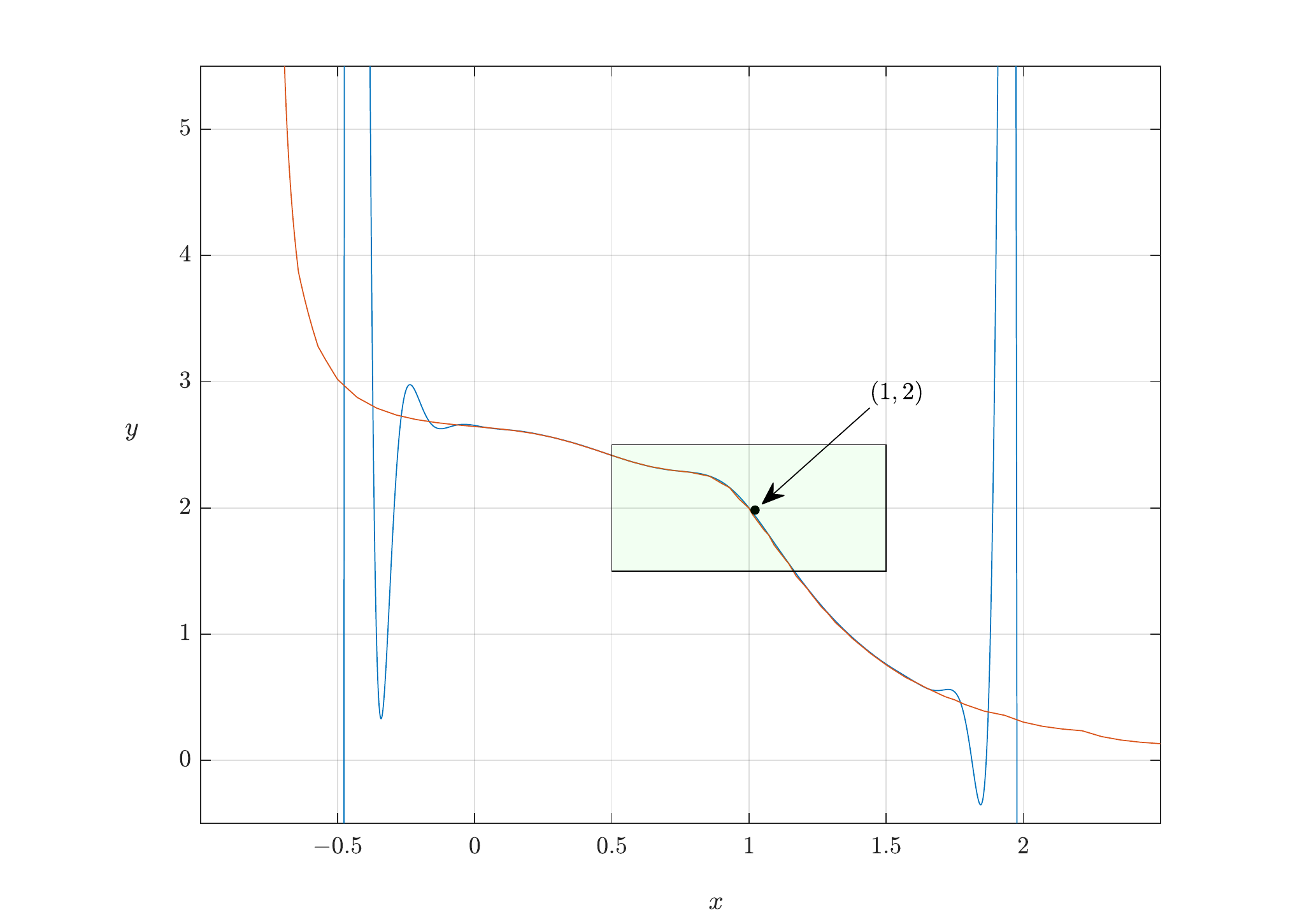}}
\caption{The function $y=h_1(x)$ is denoted in blue in $R\times[1.5,2.5]$ containing $(1,2)$, which comes from $f_1(x,y,h_2(x,y))=0$. The red indicates a graphical plot of $f_1(x,y,h_2(x,y))=0$.}
\label{fig8}
\end{figure}
\begin{figure}[!ht]
\centerline{\includegraphics[clip, trim=0cm 7.5cm 0cm 7.5cm, width=0.70\textwidth]{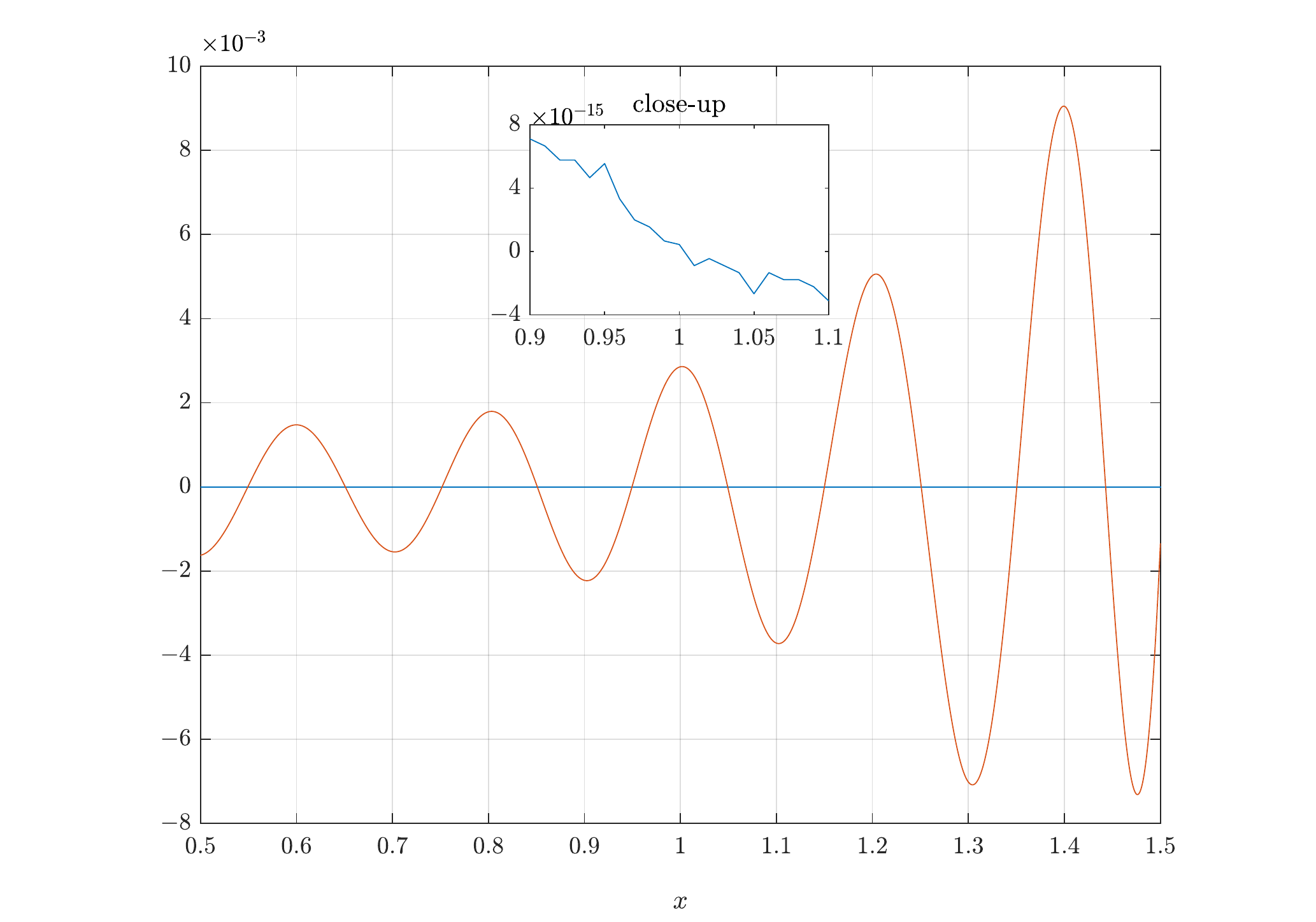}}
\caption{The blue and red curves are drawn by $f_1(x,g_1(x),g_2(x))$ and $f_2(x,g_1(x),g_2(x))$, respectively.}
\label{fig8-1}
\end{figure}
\begin{figure}[!ht]
\centerline{\includegraphics[clip, trim=0cm 7.5cm 0cm 7.5cm, width=0.70\textwidth]{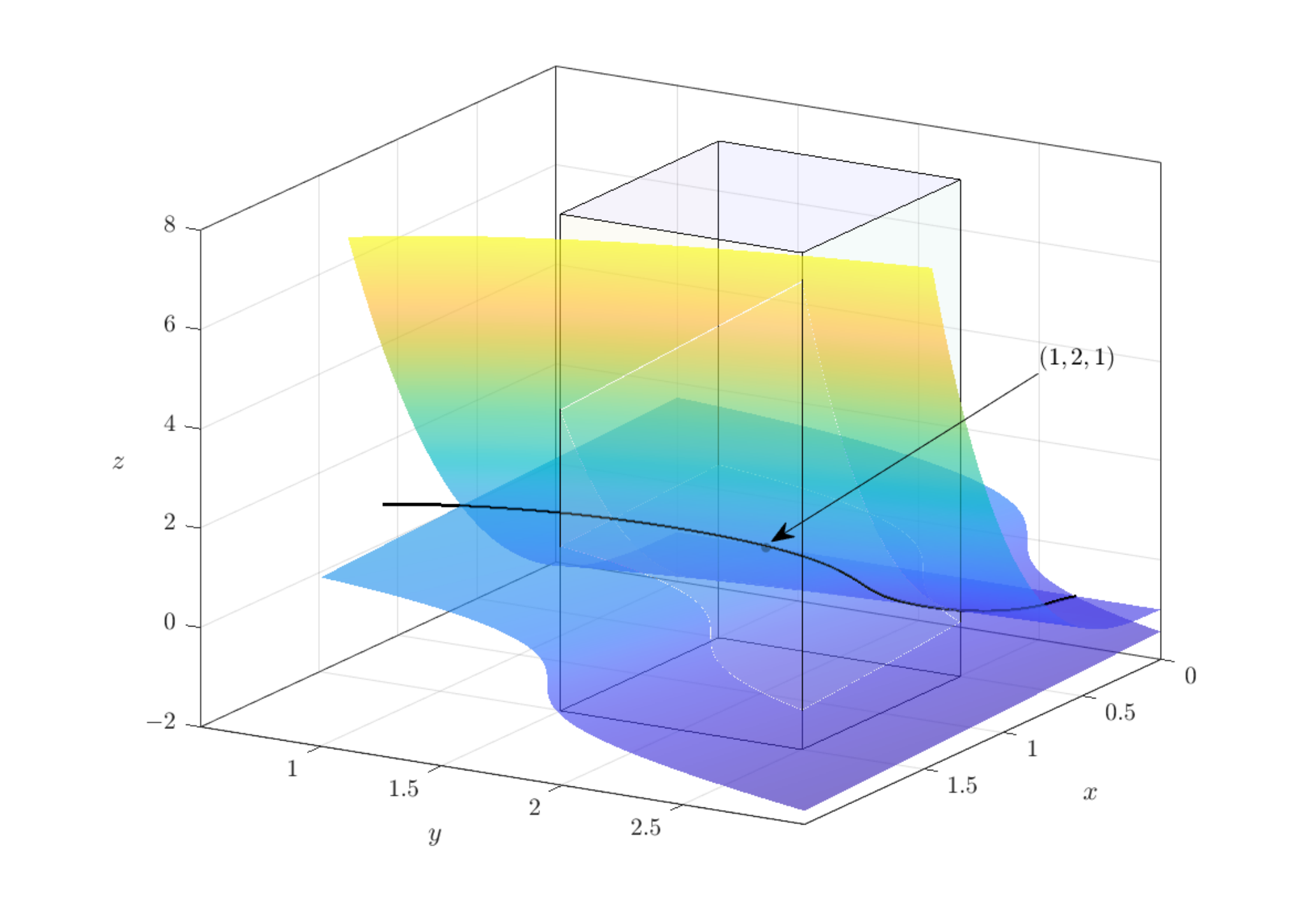}}
\vspace{-0.5cm}
\caption{The black curve denotes $x\mapsto(y,z)$, where $y=g_1(x)$ and $z=g_2(x)$, which approximate $y=y(x)$ and $z=z(x)$ such that $f_1(x,y,z)=0$ and $f_2(x,y,z)=0$ in $R\times I$, respectively.}
\label{fig9}
\end{figure}
\end{example}

The rectangle $R$ in Assumption 2 may contain a degenerate critical point. Regardless of the case, Theorem \ref{system of analytic solutions} still works.
We present such an example to show that a system of implicit functions is approximated by 
multivariate polynomials.

\begin{example} \label{dege}
Consider the pair of equations 
\begin{equation} \label{degenerate_equations}
\begin{split}
f_1(x,y,z) &= x^2 + y^2 +(x-1)z^2, \\
f_2(x,y,z) &= x^2 + 2y^2 +(y-1) z^2
\end{split}
\end{equation}
with $(0,0,0)$ at which $f_1$ and $f_2$ vanish. Since there is no linear term in (\ref{degenerate_equations}), at $(0,0,0)$ the rank of the Jacobian matrix 
is zero (as depicted in Figure \ref{fig10}). 
However, confining a range, for example, to $0\le x\le1$, $0\le y\le1$, and $0\le z<\infty$ 
and taking a rectangle $[0,0.5]\times[0,0.5]\times[0,3]$ 
which has a corner at $(0,0,0)$. 
On the rectangle, by observation, $\sign_z f_k(x,y,z)$ has only one jump discontinuity and $n_z(f_k)=-1$, $n_x(f_k)=n_y(f_k)=1$ $(k=1,2)$.

By Theorem \ref{system of analytic solutions} 
with $N=8$, $z=h_1(x,y)$ and $z=h_2(x,y)$ approximate $z=z_1(x,y)$ and $z=z_2(x,y)$ such that $f_1(x,y,z_1)=0$ and $f_2(x,y,z_2)=0$ whose coefficient matrices  are
\begin{equation*}
\arraycolsep=1.5pt\def\arraystretch{0.6}	
\begin{matrix}
& \\
\left(\kern-4pt \vphantom{ \begin{matrix} 12 \\ 12 \\ 12 \\ 12 \\ 12 \\ 12 \\ 12 \\ 12  \end{matrix} } \right.
\end{matrix}
\hspace{-0pt}
\begin{matrix} 
\Scale[0.75]{1} & \Scale[0.75]{y} & \Scale[0.75]{y^2}  & \Scale[0.75]{y^3} & 
\Scale[0.75]{y^4} & \Scale[0.75]{y^5} & \Scale[0.75]{y^6} & \Scale[0.75]{y^7} & \\ \arrayrulecolor{gray} \hline \\[-2mm]
  \Scale[0.55]{\rotatebox{0}{0.408241400667676}} &
  \Scale[0.55]{\rotatebox{0}{0.816531359019761}} &
  \Scale[0.55]{\rotatebox{0}{0.816117756871955}} &
  \Scale[0.55]{\rotatebox{0}{-1.63446077805188}} &
  \Scale[0.55]{\rotatebox{0}{2.50792658208951}} &
  \Scale[0.55]{\rotatebox{0}{-1.48883805443688}} &
  \Scale[0.55]{\rotatebox{0}{-8.02452447454164}} &
  \Scale[0.55]{\rotatebox{0}{21.2048659921929}} \\
  \Scale[0.55]{\rotatebox{0}{1.08868964456673}} &
  \Scale[0.55]{\rotatebox{0}{-1.08903120663835}} &
  \Scale[0.55]{\rotatebox{0}{2.19322111746780}} &
  \Scale[0.55]{\rotatebox{0}{2.15725284820275}} &
  \Scale[0.55]{\rotatebox{0}{-21.7192534727769}} &
  \Scale[0.55]{\rotatebox{0}{57.4809959627986}} &
  \Scale[0.55]{\rotatebox{0}{22.3087274847608}} &
  \Scale[0.55]{\rotatebox{0}{-292.817460147257}} \\
  \Scale[0.55]{\rotatebox{0}{1.63286148893869}} &
  \Scale[0.55]{\rotatebox{0}{1.10162416860967}} &
  \Scale[0.55]{\rotatebox{0}{-10.2243383275022}} &
  \Scale[0.55]{\rotatebox{0}{25.8424726647171}} &
  \Scale[0.55]{\rotatebox{0}{55.6146230887942}} &
  \Scale[0.55]{\rotatebox{0}{-522.134552371370}} &
  \Scale[0.55]{\rotatebox{0}{345.242287373544}} &
  \Scale[0.55]{\rotatebox{0}{2375.12070190542}} \\
  \Scale[0.55]{\rotatebox{0}{-0.243427885794297}} &
  \Scale[0.55]{\rotatebox{0}{3.86133661190640}} &
  \Scale[0.55]{\rotatebox{0}{18.5593250904849}} &
  \Scale[0.55]{\rotatebox{0}{-161.347274796296}} &
  \Scale[0.55]{\rotatebox{0}{318.132884269998}} &
  \Scale[0.55]{\rotatebox{0}{1280.58538056068}} &
  \Scale[0.55]{\rotatebox{0}{-5803.03015776689}} &
  \Scale[0.55]{\rotatebox{0}{4194.37168512031}} \\
  \Scale[0.55]{\rotatebox{0}{2.88860850137300}} &
  \Scale[0.55]{\rotatebox{0}{-20.2239288354349}} &
  \Scale[0.55]{\rotatebox{0}{78.8833088283334}} &
  \Scale[0.55]{\rotatebox{0}{185.824867341472}} &
  \Scale[0.55]{\rotatebox{0}{-5658.62904533275}} &
  \Scale[0.55]{\rotatebox{0}{21857.4809412704}} &
  \Scale[0.55]{\rotatebox{0}{61487.2615726907}} &
  \Scale[0.55]{\rotatebox{0}{-321312.022583376}} \\
  \Scale[0.55]{\rotatebox{0}{0.810126065016658}} &
  \Scale[0.55]{\rotatebox{0}{45.9761921136848}} &
  \Scale[0.55]{\rotatebox{0}{-519.510533394060}} &
  \Scale[0.55]{\rotatebox{0}{1708.18678563788}} &
  \Scale[0.55]{\rotatebox{0}{19624.6879190755}} &
  \Scale[0.55]{\rotatebox{0}{-127143.126595965}} &
  \Scale[0.55]{\rotatebox{0}{-153229.075397079}} &
  \Scale[0.55]{\rotatebox{0}{1382061.71528484}} \\
  \Scale[0.55]{\rotatebox{0}{-5.34208712455283}} &
  \Scale[0.55]{\rotatebox{0}{38.6432668335212}} &
  \Scale[0.55]{\rotatebox{0}{134.504396172128}} &
  \Scale[0.55]{\rotatebox{0}{-4977.91637626210}} &
  \Scale[0.55]{\rotatebox{0}{65268.6651956214}} &
  \Scale[0.55]{\rotatebox{0}{-186901.326165857}} &
  \Scale[0.55]{\rotatebox{0}{-1100349.48234761}} &
  \Scale[0.55]{\rotatebox{0}{4454278.48681176}} \\
  \Scale[0.55]{\rotatebox{0}{18.7678308525790}} &
  \Scale[0.55]{\rotatebox{0}{-276.726744898086}} &
  \Scale[0.55]{\rotatebox{0}{2817.28742085553}} &
  \Scale[0.55]{\rotatebox{0}{-246.968670877888}} &
  \Scale[0.55]{\rotatebox{0}{-314854.684257636}} &
  \Scale[0.55]{\rotatebox{0}{1459642.19080275}} &
  \Scale[0.55]{\rotatebox{0}{4208677.39377466}} &
  \Scale[0.55]{\rotatebox{0}{-22904429.3420260}} \\
 \end{matrix} 
\hspace{-0.2em}
\begin{matrix}
& \\[0.45em]
\left . \vphantom{ \begin{matrix} 12 \\ 12 \\ 12 \\ 12 \\ 12 \\ 12 \\ 12 \\ 12  \end{matrix} } \right )
\begin{matrix}
\Scale[0.65]{1} \\ \Scale[0.65]{x}  \\ \Scale[0.65]{x^2} \\ \Scale[0.65]{x^3} \\
\Scale[0.65]{x^4} \\ \Scale[0.65]{x^5} \\ \Scale[0.65]{x^6} \\ \Scale[0.65]{x^7}
\end{matrix}
\end{matrix}	%
\end{equation*}
and 
\begin{equation*}
\arraycolsep=1.5pt\def\arraystretch{0.6}	
\begin{matrix}
& \\
\left(\kern-4pt \vphantom{ \begin{matrix} 12 \\ 12 \\ 12 \\ 12 \\ 12 \\ 12 \\ 12 \\ 12  \end{matrix} } \right.
\end{matrix}
\hspace{-0pt}
\begin{matrix} 
\Scale[0.75]{1} & \Scale[0.75]{y} & \Scale[0.75]{y^2}  & \Scale[0.75]{y^3} & 
\Scale[0.75]{y^4} & \Scale[0.75]{y^5} & \Scale[0.75]{y^6} & \Scale[0.75]{y^7} & \\ \arrayrulecolor{gray} \hline \\[-2mm]
  \Scale[0.55]{\rotatebox{0}{0.499998460565297}} &
  \Scale[0.55]{\rotatebox{0}{1.66668593709208}} &
  \Scale[0.55]{\rotatebox{0}{2.10817219588402}} &
  \Scale[0.55]{\rotatebox{0}{-0.511194801196446}} &
  \Scale[0.55]{\rotatebox{0}{6.36464990241272}} &
  \Scale[0.55]{\rotatebox{0}{-7.09380895560241}} &
  \Scale[0.55]{\rotatebox{0}{-4.67859770787413}} &
  \Scale[0.55]{\rotatebox{0}{57.9648809959836}} \\
  \Scale[0.55]{\rotatebox{0}{0.666653645507011}} &
  \Scale[0.55]{\rotatebox{0}{-1.33330925771400}} &
  \Scale[0.55]{\rotatebox{0}{2.84808687276566}} &
  \Scale[0.55]{\rotatebox{0}{-1.66216527457875}} &
  \Scale[0.55]{\rotatebox{0}{-18.2420025212197}} &
  \Scale[0.55]{\rotatebox{0}{104.884228084066}} &
  \Scale[0.55]{\rotatebox{0}{-165.532336874666}} &
  \Scale[0.55]{\rotatebox{0}{-91.7951419714518}} \\  
  \Scale[0.55]{\rotatebox{0}{0.888956582839726}} &
  \Scale[0.55]{\rotatebox{0}{0.591804280528688}} &
  \Scale[0.55]{\rotatebox{0}{-8.91692976651670}} &
  \Scale[0.55]{\rotatebox{0}{42.5967553060441}} &
  \Scale[0.55]{\rotatebox{0}{-110.442206417017}} &
  \Scale[0.55]{\rotatebox{0}{-112.581487186787}} &
  \Scale[0.55]{\rotatebox{0}{2237.66959943601}} &
  \Scale[0.55]{\rotatebox{0}{-5202.74439326453}} \\  
  \Scale[0.55]{\rotatebox{0}{-1.18333322661317}} &
  \Scale[0.55]{\rotatebox{0}{5.53432073339570}} &
  \Scale[0.55]{\rotatebox{0}{-7.65765640294483}} &
  \Scale[0.55]{\rotatebox{0}{-83.2590332723340}} &
  \Scale[0.55]{\rotatebox{0}{1050.98717515900}} &
  \Scale[0.55]{\rotatebox{0}{-3369.24411323074}} &
  \Scale[0.55]{\rotatebox{0}{-10968.6904511511}} &
  \Scale[0.55]{\rotatebox{0}{52113.7334662549}} \\
  \Scale[0.55]{\rotatebox{0}{0.755515127011205}} &
  \Scale[0.55]{\rotatebox{0}{-14.1281347165804}} &
  \Scale[0.55]{\rotatebox{0}{117.141832651004}} &
  \Scale[0.55]{\rotatebox{0}{-446.923762231644}} &
  \Scale[0.55]{\rotatebox{0}{-3051.84753534045}} &
  \Scale[0.55]{\rotatebox{0}{25541.2822154715}} &
  \Scale[0.55]{\rotatebox{0}{5367.23753030004}} &
  \Scale[0.55]{\rotatebox{0}{-208929.001410320}} \\
  \Scale[0.55]{\rotatebox{0}{1.15562806583154}} &
  \Scale[0.55]{\rotatebox{0}{7.31554371402922}} &
  \Scale[0.55]{\rotatebox{0}{-250.790503985844}} &
  \Scale[0.55]{\rotatebox{0}{2273.89071435435}} &
  \Scale[0.55]{\rotatebox{0}{-5099.53557921255}} &
  \Scale[0.55]{\rotatebox{0}{-37753.2238770695}} &
  \Scale[0.55]{\rotatebox{0}{199854.775802096}} &
  \Scale[0.55]{\rotatebox{0}{-229837.302557044}} \\
  \Scale[0.55]{\rotatebox{0}{-3.42193556701504}} &
  \Scale[0.55]{\rotatebox{0}{45.9408234577807}} &
  \Scale[0.55]{\rotatebox{0}{-642.740625181014}} &
  \Scale[0.55]{\rotatebox{0}{1883.35634160718}} &
  \Scale[0.55]{\rotatebox{0}{64000.0674666749}} &
  \Scale[0.55]{\rotatebox{0}{-360910.140011759}} &
  \Scale[0.55]{\rotatebox{0}{-774407.965648020}} &
  \Scale[0.55]{\rotatebox{0}{5036249.52507314}} \\
  \Scale[0.55]{\rotatebox{0}{2.04331521739714}} &
  \Scale[0.55]{\rotatebox{0}{-95.5669077741437}} &
  \Scale[0.55]{\rotatebox{0}{2664.22271348121}} &
  \Scale[0.55]{\rotatebox{0}{-17979.3545608622}} &
  \Scale[0.55]{\rotatebox{0}{-131490.726833923}} &
  \Scale[0.55]{\rotatebox{0}{1166833.18379635}} &
  \Scale[0.55]{\rotatebox{0}{782649.427964863}} &
  \Scale[0.55]{\rotatebox{0}{-11779204.4616064}} \\
\end{matrix}
\hspace{-0.2em}
\begin{matrix}
& \\[0.45em]
\left . \vphantom{ \begin{matrix} 12 \\ 12 \\ 12 \\ 12 \\ 12 \\ 12 \\ 12 \\ 12  \end{matrix} } \right )
\begin{matrix}
\Scale[0.65]{1} \\ \Scale[0.65]{x}  \\ \Scale[0.65]{x^2} \\ \Scale[0.65]{x^3} \\
\Scale[0.65]{x^4} \\ \Scale[0.65]{x^5} \\ \Scale[0.65]{x^6} \\ \Scale[0.65]{x^7}
\end{matrix}
\end{matrix},	%
\end{equation*}
respectively.
Put $h= h_1- h_2$ and then $n_y(h)=-1$ in $[0,0.5]\times[0,0.5]$.
With $N=8$, the coefficient matrix of $y=h(x)$ is calculated as
\begin{equation*}
\arraycolsep=1.5pt\def\arraystretch{0.6}	
\begin{matrix}
& \\
\left(\kern-4pt \vphantom{ \begin{matrix} 12  \end{matrix} } \right.
\end{matrix}
\hspace{-0pt}
\begin{matrix} 
\Scale[0.75]{1} & \Scale[0.75]{x} & \Scale[0.75]{x^2}  & \Scale[0.75]{x^3} & 
\Scale[0.75]{x^4} & \Scale[0.75]{x^5} & \Scale[0.75]{x^6} & \Scale[0.75]{x^7} & \\ \arrayrulecolor{gray} \hline \\[-2mm]
  \Scale[0.58]{\rotatebox{0}{0.116319702320269}} &
  \Scale[0.58]{\rotatebox{0}{0.713165167944210}} &
  \Scale[0.58]{\rotatebox{0}{0.843601400792723}} &
  \Scale[0.58]{\rotatebox{0}{-0.552538829658902}} &
  \Scale[0.58]{\rotatebox{0}{-3.00696939156087}} &
  \Scale[0.58]{\rotatebox{0}{19.9372907280740}} &
  \Scale[0.58]{\rotatebox{0}{70.0568910222166}} &
  \Scale[0.58]{\rotatebox{0}{-382.742430825347}} 
 \end{matrix} 
\begin{matrix}
& \\
\left . \vphantom{ \begin{matrix} 12  \end{matrix} } \right ).
\end{matrix}	%
\end{equation*}
In Figure \ref{fig11}, the values of $f_1(x,h(x),h_2(x,h(x))$ and $f_2(x,h(x),h_2(x,h(x))$ are illustrated, where $y=h(x)$ and $z=h_2(x,h(x))$ approximate $y=y(x)$ and $z=z(x)$ such that  $f_1(x,y,z)=0$ and $f_2(x,y,z)=\mathbf{0}$, respectively. 
Furthermore, the approximated curve $x\mapsto(h(x),h_2(x,h(x))$ is shown in Figure \ref{fig12}.

\begin{figure}[!ht]
\centerline{\includegraphics[clip, trim=0cm 7.5cm 0cm 7.5cm, width=0.70\textwidth]{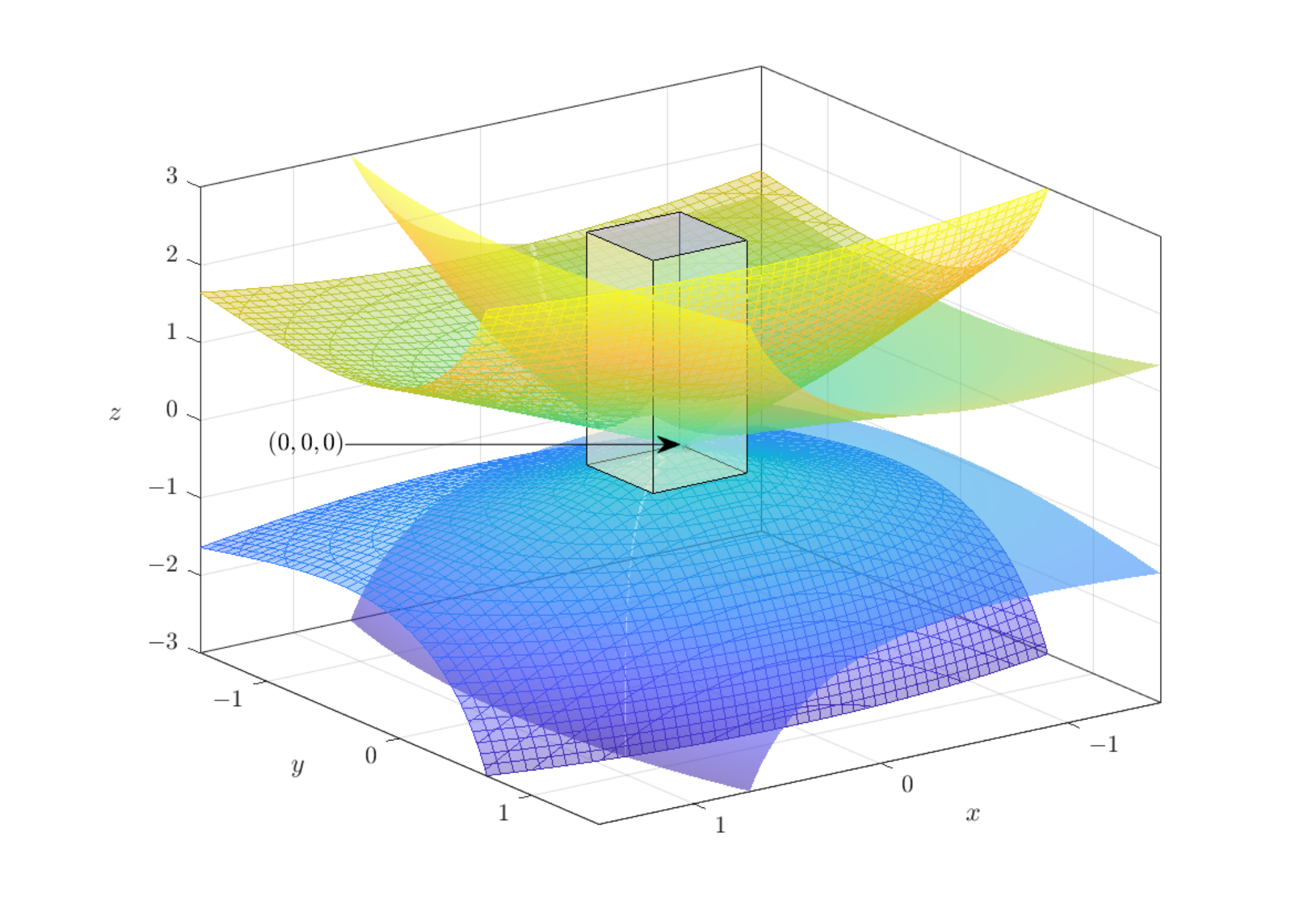}}
\vspace{-0.5cm}
\caption{Surfaces of $f_1(x,y,z)=x^2 + y^2 +(x-1)z^2=0$, $f_2(x,y,z)=x^2 + 2y^2 +(y-1) z^2=0$ with $f(0,0,0)=(0,0)$ and a rectangle $[0,0.5]\times[0,0.5]\times[0,3]$ which contains a corner $(0,0,0)$. 
The adjacent netted surfaces illustrate the implicit function $f_2(x,y,z)=0$.
 }
\label{fig10}
\end{figure}
\begin{figure}[!ht]
\centerline{\includegraphics[clip, trim=0cm 7.5cm 0cm 7.5cm, width=0.70\textwidth]{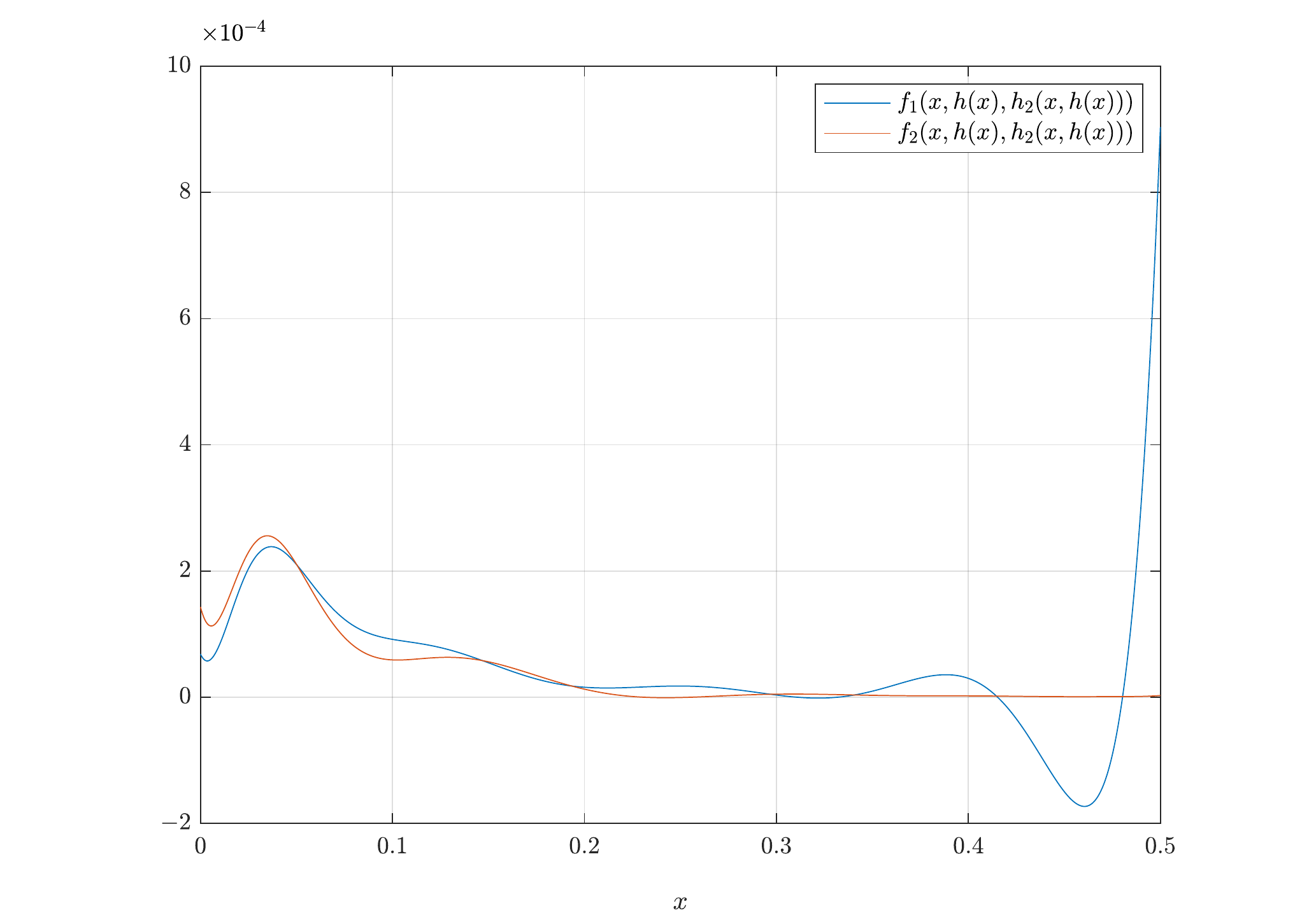}}
\caption{For $y=h(x)$ and $z=h_2(x,h(x))$,
 the values of $f_1(x,h(x),h_2(x,h(x)))$ and $f_2(x,h(x),h_2(x,h(x)))$ are illustrated.}
\label{fig11}
\end{figure}
\begin{figure}[!ht]
\centerline{\includegraphics[clip, trim=0cm 7.5cm 0cm 7.5cm, width=0.70\textwidth]{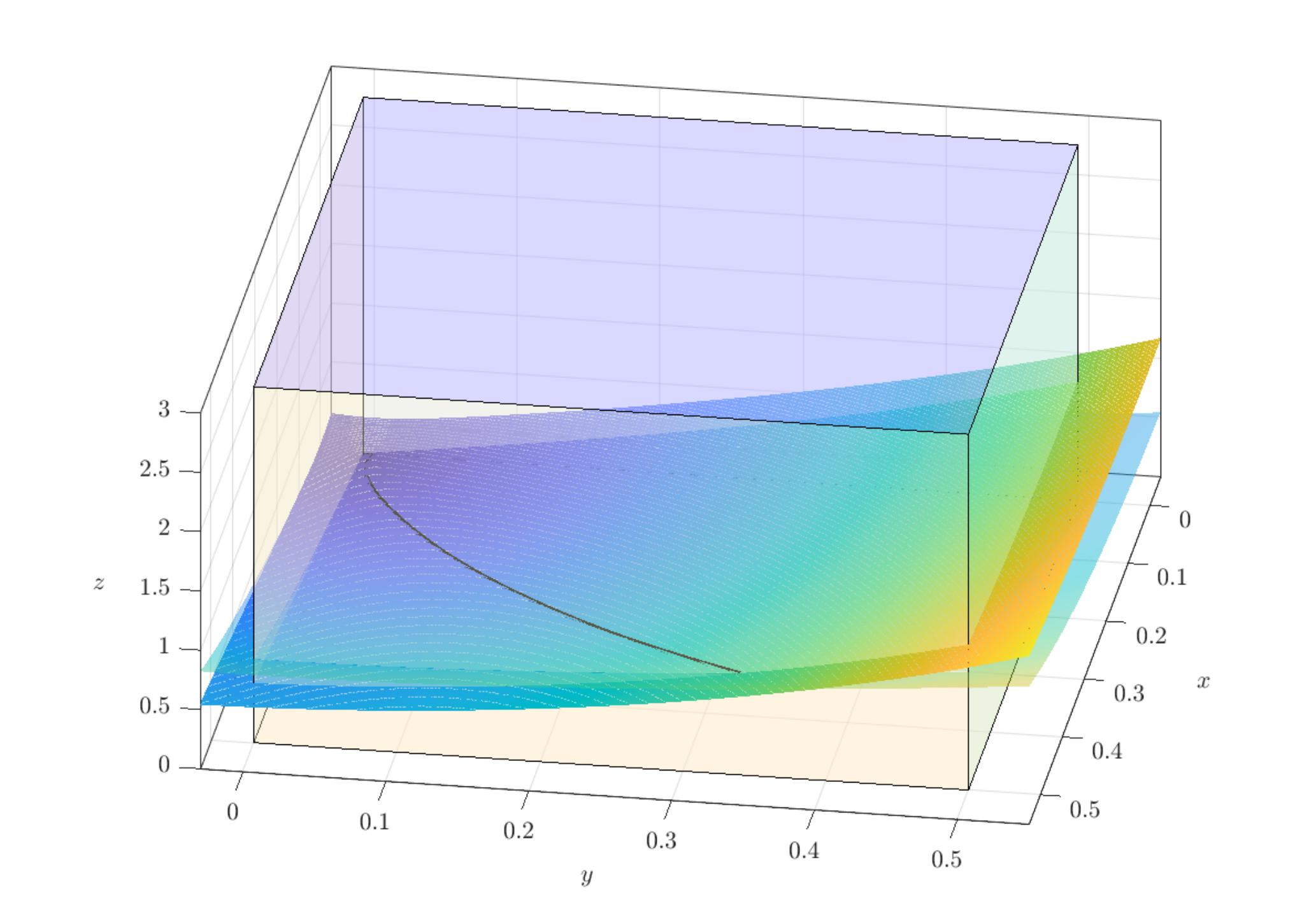}}
\caption{In $(0,0,0)\in [0,0.5]\times[0,0.5]\times[0,3]$, two surfaces of $f_1(x,y,z)=x^2 + y^2 +(x-1)z^2=0$ and $f_1(x,y,z)=x^2 + 2y^2 +(y-1) z^2=0$, and the curve of $(h(x),h_2(x,h(x))$ are depicted.}
\label{fig12}
\end{figure}

\end{example}

\begin{remark}
\begin{itemize}
\item[$(i)$ ] Throughout this article the partitions  can be chosen differently on each axis.
\item[$(ii)$] 
%

We close this article with the application of our formulation to a continuous implicit function.
For a continuous function $f(x,y):\mathbb{R}^{n+1}\to\mathbb{R}$, suppose that there is 
an $R\times I$ and a unique continuous function $y=g(x):R\to I$ such that $f=0$.
We do not know whether  
the sequence of multivariate polynomials which are calculated in Section 4, converges to the continuous implicit function.
Their coefficient matrices depend on the choice of partitions of $R$.
Because of this, by choosing the partitions suitably, the desired convergence can be expected.
More precisely, by using dyadic decomposition of partitions, the derived multivariate polynomials can converge to the continuous implicit function in weak-star topology.
Indeed, put $\Gamma_1=\{R\}$ and let $\Gamma_2$ be a collection of grid blocks $R_{2,\alpha}$ of $R$, which are caused by dyadic decomposition on $R$. 
For a positive integer $N$, inductively, let $\Gamma_{N+1}$ consist of all grid blocks $R_{2^{N+1},\alpha}$ of every element of $\Gamma_N$ by dyadic decomposition.
Now, construct a multivariate polynomial $\tilde g_N$ whose coefficients are calculate by solving (\ref{poly integration}) over $\Gamma_N$ according to (\ref{1st coefficients in poly}).  Let $R_{2^{N'},\alpha}\in \bigcup_N \Gamma_N$. 
For any $N\ge N'$, there are non-overlapping grid blocks in $\Gamma_N$ whose union equals $R_{2^{N'},\alpha}$, i.e., $R_{2^{N'},\alpha}=\bigcup_{R_{2^{N},\beta}\subset R_{2^{N'},\alpha}}R_{2^{N},\beta}$.
Furthermore, the grid points of $\Gamma_{N'}$ are contained in those of $\Gamma_{N}$.
By (\ref{poly integration}) and (\ref{2nd data}), we have  
\begin{equation} \label{finer}
	\int_{R_{2^{N'},\alpha}}\tilde g_N-g\,dx
	=\sum_{R_{2^{N},\beta}\subset R_{2^{N'},\alpha}}
		\int_{R_{2^{N},\beta}}\tilde g_N-g\,dx=0
\end{equation}
for every $N\ge N'$.
On the other hand, since the collection of all finite linear combinations of characteristic functions which are supported in dyadic grid blocks in $\bigcup_N \Gamma_N$ is dense in $L^1(R)$, the multivariate polynomial $\tilde g_N$ converges to $g$ weak-star in $L^\infty(R)$ as $N\to\infty$ by the Riesz representation theorem for the Lebesgue spaces.
%
\end{itemize}
\end{remark}




\end{document}